\documentclass[12pt]{amsart}

\usepackage{latexsym, amssymb}

\newcommand{\be}{\begin{equation}}
\newcommand{\ee}{\end{equation}}
\newcommand{\beq}{\begin{eqnarray}}
\newcommand{\eeq}{\end{eqnarray}}

\newtheorem{prop}{Proposition}[section]
\newtheorem{thm}[prop]{Theorem}
\newtheorem{lem}[prop]{Lemma}

\newtheorem{defi}[prop]{Definition}
\newtheorem{con}[prop]{Conjecture}

\def\begeq{\begin{equation}}
\def\endeq{\end{equation}}
\def\p{\partial}

\begin{document}
\title {Volume comparison of conformally compact manifolds with scalar curvature $R\geq -n\left(n-1\right)$ }

\begin{abstract}
In this paper, we use the normalized Ricci-DeTurk flow to prove a stability result for strictly stable conformally compact Einstein manifolds. As an application, we show a local volume comparison of conformally compact manifolds with scalar curvature $R\geq -n\left(n-1\right)$ and also the rigidity result when certain renormalized volume is zero.\\

Dans cet article, nous utilisons le  flot de Ricci-DeTurk normalis\'e  pour prouver la stabilit\'e des vari\'et\'es d'Einstein strictement stables et  conform\'ement compactes. En tant qu'application, nous montrons une comparaison de volume local pour les vari\'et\'es conform\'ement compactes dont la courbure scalaire satisfait $R \geq -n(n-1)$. Nous donnons \'egalement un r\'esultat de rigidit\'e lorsque certain volume renormalis\'e est nul.
\end{abstract}

\renewcommand{\subjclassname}{\textup{2000} Mathematics Subject Classification}

\author{Xue Hu $^\dag$,  Dandan Ji $^\ddag$
and  Yuguang Shi$^\ddag$}

\address{Xue Hu, Academy of Mathematics and Systems Science, Chinese Academy of Sciences,
Beijing, 100190, P.R. China.} \email{huxue@amss.ac.cn}

\address{Dandan Ji, Key Laboratory of Pure and Applied mathematics, School of Mathematics Science, Peking University,
Beijing, 100871, P.R. China.} \email{jidandan@pku.edu.cn}

\address{Yuguang Shi, Key Laboratory of Pure and Applied mathematics, School of Mathematics Science, Peking University,
Beijing, 100871, P.R. China.} \email{ygshi@math.pku.edu.cn}

\thanks{$^\dag$ Research partially supported by China Postdoctoral Science Foundation grant 2012M520414.}
\thanks{$^\ddag$ Research partially supported by NSF grant of China 10990013.}

\date{May, 2014}
\maketitle

\markboth{}{}
\section{Introduction}
Volume is one of the natural geometric quantities which is often used to explore geometrical and topological properties of a Riemannian manifold. Classical examples in this direction are various volume comparison theorems which turned out to be fruitful in Riemannian geometry. In order to use those volume comparison theorems efficiently, we have to assume certain lower bound on the Ricci curvature of the manifold. Obviously, we can not expect the same results still to be true if we only assume lower bound on scalar curvature. However, in \cite{S}, R.Schoen proposed the following conjecture on the volume functional $vol(\cdot)$ on a closed hyperbolic manifold.
\begin{con}
Let $(M^n,\tilde{g})$ be a closed hyperbolic manifold. Let $g$ be another metric on $M$ with scalar curvature $R(g)\geq R(\tilde{g})$, then
$vol(g)\geq vol(\tilde{g})$.
\end{con}

This conjecture remained widely open until its $3$-dimensional case followed as a corollary of Perelman's work on geometrization (\cite{P1,P2}). Later, Agol, Storm and Thurston established similar results on compact hyperbolic $3$-manifolds whose boundary are minimal surfaces (\cite{AST}). In \cite{MT}, Miao and Tam found that the above conjecture on closed manifolds does not generalize directly to manifolds with boundary if
only the Dirichlet boundary condition is imposed. Precisely, they gave a negative answer on the geodesic balls in $3$-dimensional hyperbolic space by using the variation of the volume with scalar curvature constraints. It is natural to ask if there
exists a similar conjecture or result on noncompact manifolds with hyperbolic metric, more generally, a kind of conformally compact manifold. However, what needs to be generalized first is the concept of an appropriate "volume". Recently, inspired by Bray's thesis \cite{Br}, Brendle and Chodosh addressed a notion of renormalized volume for some asymptotically Anti-deSitter-Schwarzschild manifolds in \cite{BC} and established an interesting volume comparison result with a lower bound on the scalar curvature. Their result only works for $3$-dimensional asymptotically Anti-deSitter-Schwarzschild manifolds because they need to use Geroch monotonicity formula for the Hawking mass along the IMCF where the issue of Gauss-Bonnet formula played an important role. With these results in mind, it may be natural to ask whether or not a version of Schoen's Conjecture in terms of this renormalized volume introduced in \cite{BC} is true for arbitrary dimensional conformally compact manifolds. As the first step, we will give an affirmative answer to this when the manifold is a slight perturbation of hyperbolic space or other strictly stable conformal compact Einstein manifold. We point out that in \cite{BMW}, a similar notion of renormalized volume was introduced and an interesting monotonicity of it was established under Ricci flow.

In order to state our main results, let us recall some basic definitions. Suppose that $\bar{M}^{n}$ is a smooth manifold with the boundary $\partial M$ and the interior $M^n$. A defining function $x$ of the boundary in
$M^{n}$ is a smooth function on $\bar{M}^{n}$ such that $x> 0$ in $M^{n}$; $x= 0$ on $\p M$; $dx \neq 0$ on $\p M$. A complete noncompact Riemannian metric $g$ on $M$ is said to be
conformally compact of regularity $C^{k, \mu}$ if $x^2g$ extends to
be a $C^{k, \mu}$  Riemannian metric on $\bar M$. The metric $x^2g$ induces a metric $\hat{g}$ on the boundary $\p M,$ and the metric
$g$ induces a conformal class of metric $[\hat{g}]$ on the boundary $\p M$ when defining functions
vary. The conformal manifold $(\p M, [\hat{g}])$ is called the conformal infinity of the conformally
compact manifold $(M,g)$.

Let $\tilde{g}$ be an arbitrary conformally compact metric and $g$ be a complete noncompact Riemannian metric on $M^n.$ As the same in \cite{BC}, we define the renormalized volume of $g$ with respect to $\tilde{g}$ by
$$\mathcal{V}_{\tilde{g}}(g):=\lim_{i\rightarrow \infty}\left(vol\left(\Omega_i, g\right)-vol\left(\Omega_i, \tilde{g}\right)\right)$$
where $\Omega_i$ is an arbitrary exhaustion of $M$ by compact sets. It is easy to see that if $g$ has sufficient decay relative to $\tilde{g},$ say, $\|e^{\tau\rho}\left(g-\tilde{g}\right)\|_{C^0\left(M,\tilde{g}\right)}=O\left(1\right),$ where $\tau>n-1$ and $\rho$ is the distance function to some fixed point or set in $M$ with respect to $\tilde{g},$ then $\mathcal{V}_{\tilde{g}}(g)$ is well-defined. In this paper, we are interested in the case that the ground state metric $\tilde{g}$ is a $C^{2,\alpha}$ strictly stable conformally compact Einstein metric. Please see Definition \ref{strictlystable}.

However, in order to avoid the complexity of the end structure of conformally compact manifold, we always need the concept of an essential set. Please see Definition 1.1 in \cite{HQS} for an essential set. In \cite{G} Lemma 2.5.11 and Corollary 2.5.12, Gicquaud proved that if a complete noncompact manifold is $C^2$ conformally compact then it contains essential sets. Once a conformally compact manifold $\left(M^n,g\right)$ has an essential set $\mathbb{D}$, the volume $vol(B_{g}(p,1))$ has a lower bound $\Lambda$$=\Lambda\left(g,n\right)$ for all $p\in M$ where $B_{g}(p,1)\subset M$ is a geodesic ball of radius $1$ at center $p.$ It follows immediately from the well known result (see Lemma 3.1 in \cite{QSW}) that there exists some $i=i\left(\Lambda,k,n\right)>0$ such that the injectivity radius of $(M,g)$ satisfies $inj_{\left(M,g\right)}\geq i$ provided $\|Rm\|_{C^0\left(M^n,g\right)}\leq k.$

Let $\left(M^n,g\right)$ be a $C^{2,\alpha}$ conformally compact Einstein manifold and $L=\Delta_{L}+2\left(n-1\right)$ denote the linearization of the curvature operator $Ric+\left(n-1\right)g$ where $\Delta_{L}$ is the Lichnerowicz Laplacian with respect to the metric $g$ on symmetric $2$-tensor $u,$ i.e.,
\begin{equation}
\Delta_{L}u_{ij}=-\Delta u_{ij}-2R_{ipjq}u^{pq}+R_{iq}u_{j}^{q}+R_{jq}u_{i}^{q}.\nonumber
\end{equation}
It follows from work of Delay (\cite{D}) and Lee (\cite{Lee} Proposition D) that the essential $L^2$ spectrum of $\Delta_{L}$ is $[\frac{\left(n-1\right)^2}{4}-2\left(n-1\right),\infty).$ Hence it is an issue about the discrete eigenvalues below the continuous spectrum of $L$ which has a strictly positive bottom.

\begin{defi}\label{strictlystable}
A $C^{2,\alpha}$ conformally compact Einstein manifold $\left(M^n, g\right)$ is called strictly stable if
\begin{equation}
\lambda=\inf_{u}\frac{\int_{M}\langle Lu,u\rangle_{g} d\mu_{g}}{\int_{M} |u|^2_{g}d\mu_{g}}>0\nonumber
\end{equation}
where the infimum is taken among all nonzero symmetric $2$-tensors $u$ such that
\begin{equation}
\int_{M}\left(|u|^2_{g}+|\nabla u|^2_{g}\right)d\mu_{g}<\infty\nonumber
\end{equation}
\end{defi}
It is known that many conformally compact Einstein manifolds are strictly stable, for example, any $C^{2,\alpha}$ conformally compact Einstein manifold either with nonpositive sectional curvature or with sectional curvature having certain upper bound and nonnegative Yamabe invariant of its conformal infinity (\cite{Lee} Theorem A). In particular, the hyperbolic space is a strictly stable conformally compact Einstein manifold.

Our first main result is an investigation of the stability of a strictly stable conformally compact Einstein manifold under normalized Ricci-DeTurck flow which will be used as a tool to prove the volume comparison result. First, we introduce normalized Ricci flow (NRF for short in the sequel),
\begin{equation}\label{Ricciflow}
\left\{
  \begin{array}{ll}
    \frac{\p}{\p t}g_{ij}=-2(R_{ij}+(n-1)g_{ij}) ~~on~~M^n\times(0,T),\\
    g\left(\cdot,0\right)=g.
  \end{array}
\right.
\end{equation}

Since NRF equation is degenerate parabolic equation, we consider the normalized Ricci-DeTurck flow (NRDF for short in the sequel) with background metric $\tilde{g}$:
\begin{equation}\label{RicciDeTurckflow}
\left\{
  \begin{array}{ll}
    \frac{\p}{\p t}g_{ij}=-2(R_{ij}+(n-1)g_{ij})+\nabla_{i}V_j+\nabla_{j}V_i ~~\text{on}~~M^n\times(0,T),\\
    g\left(\cdot,0\right)=g.
  \end{array}
\right.
\end{equation}
where $V_{j}=g_{jk}g^{pq}\left(\Gamma_{pq}^{k}-\tilde{\Gamma}_{pq}^{k}\right)$ and $\tilde{\Gamma}$ are the Christoffel symbols of the metric $\tilde{g}.$

Assume that $\Phi_t:M^n\longrightarrow M^n$ solves
\begin{equation}\label{diffeomorphism}
\left\{
  \begin{array}{ll}
    \frac{\p}{\p t}\Phi_t\left(x\right)=-V\left(\Phi_t\left(x\right),t\right),\\
    \Phi_0\left(x\right)=id\left(x\right),
  \end{array}
\right.
\end{equation}
where the components of $V$ are given by $V^{i}:=g^{ij}V_j$, then we obtain a family of smooth diffeomorphisms $\Phi_t$ for $t>0$ such that if $g(t)$, $t\in[0,T)$ is a solution to NRDF (\ref{RicciDeTurckflow}), $\bar{g}(t):=\Phi^{\ast}_{t}g(t)$, $t\in[0,T)$ is a solution to NRF (\ref{Ricciflow}).

There are several papers which investigated the stability of hyperbolic space under NRF \cite{LY,SSS,Ba,Bam}. In \cite{SSS}, Schn$\ddot{u}$rer, Schulze and Simon used the NRDF to get the stability of hyperbolic space. To be precise, under the assumptions that $\|g-h\|_{C^0\left(\mathbb{H}^n,h\right)}\leq \epsilon$ and $\|g-h\|_{L^2\left(\mathbb{H}^n,h\right)}\leq K,$ the NRDF starting from $g$ with background metric being hyperbolic metric $h$ exists globally. Moreover, there exists a constant $C=C\left(n,K\right)>0$ such that $$\|g(t)-h\|_{C^0\left(\mathbb{H}^n,h\right)}\leq C \cdot e^{-\frac{1}{4\left(n+2\right)}t}$$ for all $t\in[0,\infty),$ together with some interior estimates and interpolation, which implies that NRDF converges to $h$ exponentially in $C^k$ as $t\rightarrow\infty$ for all $k\in\mathbb{N}.$ In \cite{Ba}, Bahuaud proved that given a smoothly conformally compact asymptotically hyperbolic metric there is a short-time solution to the Ricci flow that remains smoothly conformally compact and asymptotically hyperbolic. After adapting the work of Schn$\ddot{u}$rer, Schulze and Simon, Bahuaud has used some $\eta-$admissible Einstein metric to replace the hyperbolic space in \cite{SSS} and obtained the stability of this $\eta-$admissible Einstein metric under NRF.

In this paper we generalize the above work to the stability of certain strictly stable conformally compact Einstein manifold.

\begin{thm}\label{stability}
Suppose that $\left(M^n, \tilde{g}\right)$ is a $C^{2,\alpha}$ strictly stable conformally compact Einstein manifold and $g$ is another complete noncompact Riemannian metric on $M^n.$ For all $n\geq3$ and $K>0,$ there exists some $\epsilon=\epsilon\left(n,K,\tilde{g}\right)>0$ such that if $$\|g-\tilde{g}\|_{C^0\left(M^n,\tilde{g}\right)}\leq \epsilon~~~~\text{and}~~~~\|g-\tilde{g}\|_{L^2\left(M^n,\tilde{g}\right)}\leq K,$$ then the normalized Ricci-DeTurck flow (\ref{RicciDeTurckflow}) with background metric $\tilde{g}$ starting from $g$ exists globally and converges exponentially to $\tilde{g}.$
\end{thm}
We will use this result as a tool to investigate the behavior of the renormalized volume. We are able to show:
\begin{thm}\label{main1}
Suppose that $\left(M^n, \tilde{g}\right)$ is a $C^{2,\alpha}$ strictly stable conformally compact Einstein manifold and $g$ is another complete noncompact Riemannian metric on $M^n.$ For all $n\geq4$ and $\tau>n-1,$ there exists some $\epsilon=\epsilon\left(n,\tilde{g}\right)>0$ such that if
$$\|e^{\tau\rho}\left(g-\tilde{g}\right)\|_{C^1\left(M^n,\tilde{g}\right)}\leq \epsilon,$$
and
$$R\left(g\right)\geq-n\left(n-1\right),$$
where $\rho=d_{\tilde{g}}(\cdot,\mathbb{D})$ is the distance function to some essential set $\mathbb{D}\subset M$ with respect to $\tilde{g},$ then
\begin{equation}
\mathcal{V}_{\tilde{g}}(g)\geq0.\nonumber
\end{equation}
\end{thm}

\begin{thm}\label{rigidity}
Under the assumptions of Theorem \ref{main1}, suppose that $\mathcal{V}_{\tilde{g}}(g)=0.$ Then there exists a $C^{\infty}$ diffeomorphism $\Phi:M\longrightarrow M,$ such that $g=\Phi^{\ast}\tilde{g}.$ Moreover, $\Phi$ extends continuously to some diffeomorphism on $\bar{M}$ and $\Phi|_{\p M}=id.$
\end{thm}

In particular,
\begin{thm}\label{main2}
Suppose that $\left(\mathbb{B}^n, h\right)$ is the Poincar$\acute{e}$ ball model for hyperbolic space and $g$ is another complete noncompact Riemannian metric on $\mathbb{B}^n$. For all $n\geq4$ and $\tau>n-1,$ there exists some $\epsilon=\epsilon\left(n\right)>0$ such that if
$$\|e^{\tau\rho}\left(g-h\right)\|_{C^1\left(\mathbb{B}^n,h\right)}\leq \epsilon,$$
and
$$R\left(g\right)\geq-n\left(n-1\right),$$
where $\rho=d_{h}(\cdot,\mathbb{D})$ is the distance function to some essential set $\mathbb{D}\subset \mathbb{B}^n$ with respect to $h,$ then
\begin{equation}
\mathcal{V}_{h}(g)\geq0.\nonumber
\end{equation}
When $\mathcal{V}_{h}(g)=0,$ there exists a $C^{\infty}$ diffeomorphism $\Phi:\mathbb{B}^n\longrightarrow \mathbb{B}^n,$ such that $g=\Phi^{\ast}h.$ Moreover, $\Phi$ extends continuously to some diffeomorphism on $\bar{\mathbb{B}}^n,$ and $\Phi|_{\mathbb{S}^{n-1}}=id.$
\end{thm}

The above results can be regarded as a version of Schoen's Conjecture in the case of conformally compact manifolds. Indeed, we are able to show that our renormalized volume is non-increasing along the flow (see Proposition \ref{monotonicityvolume}) and will converge to zero as time goes to infinity. In order to get these properties, we need to estimate on the gauge and the scalar curvature (see Lemma \ref{decayofV} and Lemma \ref{decayofscalarcurvature}).

Recently, Bahuaud, Mazzeo and Woolgar have studied in \cite{BMW} the evolution of another renormalized volume functional for asymptotically
Poincar$\acute{e}$-Einstein metrics which are evolving by normalized Ricci flow. And monotonicity of that quantity along the normalized Ricci flow was also obtained in their paper. However, we establish a long-time existence and convergence theorem for the normalized Ricci-DeTurck flow for metrics near a strictly stable conformally compact Einstein manifold (see Theorem \ref{stability}), hence we are able to get Theorem\ref{main1}, Theorem\ref{rigidity} and Theorem\ref{main2}.

This paper is organized as follows. In Sect.2 we will show the long-time existence and convergence of NRDF and prove Theorem \ref{stability}. In Sect.3 we give some basic estimates first and then prove our main results Theorem \ref{main1} and Theorem \ref{rigidity}. We will always omit the subscript $\tilde{g}$ in the renormalized volume functional. We will use $\|\cdot\|_{L^p},$ $\|\cdot\|_{C^k}$ for the global norm, $|\cdot|$ and $\langle\cdot,\cdot\rangle$ for the pointwise norm and inner product, $d\mu$ for the volume element, and $\rho$ for the distance function. And all these quantities are with respect to $\tilde{g}$ unless otherwise stated. We denote by $\epsilon_{i},$ $i=1,2\ldots$ some positive constants depending only on $\epsilon,$ $n$ and $\tilde{g}$ which can be arbitrary small as $\epsilon$ goes to zero, and $C_{j},$ $j=1,2\ldots,$ some positive constants depending only on $\epsilon,$ $n$ and $\tilde{g}$ which are big and bounded.

{\bf Acknowledgments} The authors are grateful to Professor Jie Qing and Dr. Romain Gicquaud for their interests in this work and helpful
conversations, and they also would like to thank the referees for numerous suggestions which helps to improve the presentation.

\section{Long-Time Existence and Convergence of NRDF}

In this section, we will prove Theorem \ref{stability}. Since $\|g-\tilde{g}\|_{C^0\left(M^n,\tilde{g}\right)}\leq \epsilon,$ due to Shi \cite{Shi}, there exists some $T=T\left(n,\tilde{g}\right)$ and a family of metrics $g\left(t\right)$ for $t\in[0,T)$ which solves the NRDF (\ref{RicciDeTurckflow}) starting from $g=g_0.$ In addition,
\begin{equation}\label{smalltimemetricdiff}
\|g\left(t\right)-\tilde{g}\|_{C^0\left(M^n,\tilde{g}\right)}\leq 2\epsilon,
\end{equation}
\begin{equation}\label{smalltimeonederivative}
\|\tilde{\nabla}g(t)\|_{C^0\left(M^n,\tilde{g}\right)}\leq \frac{\epsilon_1}{\sqrt{t}},
\end{equation}
and
\begin{equation}\label{smalltimetwoderivative}
\|\tilde{\nabla}^{2}g(t)\|_{C^0\left(M^n,\tilde{g}\right)}\leq\frac{\epsilon_2}{t},
\end{equation}
for all $t\in[0,T)$.

Let $u\left(t\right)=g\left(t\right)-\tilde{g}.$ Let $g$ denote $g\left(t\right),$ $u$ denote $u\left(t\right)$ and $g_0$ denote the initial data $g.$ Due to the assumption, $$\|g_0-\tilde{g}\|_{L^2\left(M^n,\tilde{g}\right)}\leq K.$$ In the following we compute some evolution equations under the NRDF.
\begin{lem}\label{evolutionequ}
Let $g(t)$, $t\in[0,T)$ be a solution to NRDF (\ref{RicciDeTurckflow}). Then for all $t\in[0,T),$
 \begin{equation}
 \begin{split}
\frac{\p}{\p t}|u|^2 \leq& g^{ij}\tilde{\nabla}_{i}\tilde{\nabla}_{j}|u|^2-\left(2-\epsilon_3\right)|\tilde{\nabla}g|^2
+\left(4+4|\tilde{W}|+\epsilon_4\right)|u|^2\\
\frac{\p}{\p t}|\tilde{\nabla}g|^{2}=&
g^{ij}\tilde{\nabla}_{i}\tilde{\nabla}_{j}|\tilde{\nabla}g|^{2}-2g^{ab}\tilde{g}^{mn}\tilde{g}^{ik}\tilde{g}^{jl}\tilde{\nabla}_a\tilde{\nabla}_ng_{ij}\cdot\tilde{\nabla}_b\tilde{\nabla}_mg_{kl}\\
&+\tilde{\nabla}g\ast\tilde{\nabla}g\ast\tilde{\nabla}^2g+\tilde{\nabla}g\ast\tilde{\nabla}g\ast\tilde{\nabla}g\ast\tilde{\nabla}g
+\tilde{\nabla}g\ast\tilde{\nabla}g\\
\frac{\p}{\p t}|\tilde{\nabla}^{2}g|^{2}= & g^{ij}\tilde{\nabla}_{i}\tilde{\nabla}_{j}|\tilde{\nabla}^{2}g|^{2}-2g^{ab}\tilde{g}^{mn}\tilde{g}^{ik}\tilde{g}^{jl}\tilde{g}^{pq}\tilde{\nabla}_a\tilde{\nabla}_m\tilde{\nabla}_pg_{ij}\cdot\tilde{\nabla}_b\tilde{\nabla}_n\tilde{\nabla}_qg_{kl}\\
&+\tilde{\nabla}g\ast\tilde{\nabla}^2g\ast\tilde{\nabla}^3g+\tilde{\nabla}^2 g\ast\tilde{\nabla}^2g\ast\tilde{\nabla}^2g+\tilde{\nabla}^2 g\ast\tilde{\nabla}^2g\ast\tilde{\nabla}g\ast\tilde{\nabla}g\\
&+\tilde{\nabla}^2 g\ast\tilde{\nabla}^2g+\tilde{\nabla}^2g\ast\tilde{\nabla}g\ast\tilde{\nabla}g\ast\tilde{\nabla}g\ast\tilde{\nabla}g+\tilde{\nabla}^2g\ast\tilde{\nabla}g\ast\tilde{\nabla}g\\
\end{split}\nonumber
\end{equation}
where $\ast$ denotes linear combinations with $g(t)$, $\tilde{g}$ and their inverse.
\end{lem}

\begin{proof}
Since the background metric is Einstein,
\begin{equation}
\tilde{R}_{ijkl}=\tilde{W}_{ijkl}-\left(\tilde{g}_{ik}\tilde{g}_{jl}-\tilde{g}_{il}\tilde{g}_{jk}\right)\nonumber
\end{equation}
satisfying $\tilde{g}^{pq}\tilde{W}_{ipjq}=0.$
Then a metric $g$ solving the NRDF (\ref{RicciDeTurckflow}) fulfills
\begin{equation}
\begin{split}
&\frac{\p}{\p t}g_{ij}=g^{ab}\tilde{\nabla}_{a}\tilde{\nabla}_{b}g_{ij}-2g_{ij}g^{kl}\left(g_{kl}-\tilde{g}_{kl}\right)+2\left(g_{ij}-\tilde{g}_{ij}\right)+\frac{1}{2}g^{ab}g^{pq}\\
&\cdot\left(\tilde{\nabla}_{i}g_{pa}\tilde{\nabla}_{j}g_{qb}+2\tilde{\nabla}_{a}g_{jp}\tilde{\nabla}_{q}g_{ib}-2\tilde{\nabla}_{a}g_{jp}\tilde{\nabla}_{b}g_{iq}-2\tilde{\nabla}_{j}g_{pa}\tilde{\nabla}_{b}g_{iq}-2\tilde{\nabla}_{i}g_{pa}\tilde{\nabla}_{b}g_{jq}\right)\\
&-\left(g^{kl}-\tilde{g}^{kl}\right)\left(g_{ip}-\tilde{g}_{ip}\right)\tilde{g}^{pq}\tilde{W}_{jkql}-\left(g^{kl}-\tilde{g}^{kl}\right)\left(g_{jp}-\tilde{g}_{jp}\right)\tilde{g}^{pq}\tilde{W}_{ikql}\\
&-2\left(g^{kl}-\tilde{g}^{kl}\right)\tilde{W}_{ikjl}.
\end{split}\nonumber
\end{equation}
By direct computation, we can get the other two evolution equations. Note that when the dimension $n=3,$ $\tilde{W}\equiv0.$
\end{proof}

We now rewrite the evolution equations of $u\left(t\right)=g\left(t\right)-\tilde{g}$ to which the strictly stable assumption can be related.

\begin{lem}\label{nondegenerateequ}
Let $g(t)$, $t\in[0,T)$ be a solution to the NRDF (\ref{RicciDeTurckflow}). Then for all $t\in[0,T),$
\begin{equation*}
\frac{\p}{\p t}|u|^2\leq-2\langle\tilde{L}u,u\rangle+\left(g^{ab}-\tilde{g}^{ab}\right)\tilde{\nabla}_{a}\tilde{\nabla}_{b}|u|^2+\epsilon_5\left(|\tilde{\nabla}u|^2+|u|^2\right).
\end{equation*}
\end{lem}
\begin{proof}
Since $\tilde{R}_{ij}+(n-1)\tilde{g}_{ij}=0$ and $\tilde{V}_i=0,$ then $u\left(t\right)$ fulfills
\begin{equation}
\begin{split}
&\frac{\p}{\p t}u_{ij}\\
&=-2(R_{ij}+(n-1)g_{ij})+\nabla_{i}V_j+\nabla_{j}V_i-\left(-2(\tilde{R}_{ij}+(n-1)\tilde{g}_{ij})+\nabla_{i}\tilde{V}_j+\nabla_{j}\tilde{V}_i\right)\\
&=-\left(\tilde{\Delta}_L+2\left(n-1\right)\right)u_{ij}+\mathcal{F}\left(g,\tilde{g},u\right)\\
\end{split}\nonumber
\end{equation}
where
\begin{equation}\label{F}
\begin{split}
&\mathcal{F}\left(g,\tilde{g},u\right)\\
=&-2\tilde{R}_{ikjl}u^{kl}-g^{kl}g_{ip}\tilde{g}^{pq}\tilde{R}_{jkql}-g^{kl}g_{jp}\tilde{g}^{pq}\tilde{R}_{ikql}+\tilde{R}_{ik}u_{j}^{k}
+\tilde{R}_{jk}u_{i}^{k}+\frac{1}{2}g^{ab}g^{pq}\\
&\cdot\left(\tilde{\nabla}_{i}g_{pa}\tilde{\nabla}_{j}g_{qb}+2\tilde{\nabla}_{a}g_{jp}\tilde{\nabla}_{q}g_{ib}
-2\tilde{\nabla}_{a}g_{jp}\tilde{\nabla}_{b}g_{iq}-2\tilde{\nabla}_{j}g_{pa}\tilde{\nabla}_{b}g_{iq}
-2\tilde{\nabla}_{i}g_{pa}\tilde{\nabla}_{b}g_{jq}\right)\\
&+\left(g^{ab}-\tilde{g}^{ab}\right)\tilde{\nabla}_{a}\tilde{\nabla}_{b}g_{ij}-2\left(n-1\right)\tilde{g}_{ij}.\\
\end{split}\nonumber
\end{equation}

Hence
\begin{equation}
\mathcal{F}\left(g,\tilde{g},u\right)=\left(g^{ab}-\tilde{g}^{ab}\right)\tilde{\nabla}_{a}\tilde{\nabla}_{b}g_{ij}+
\tilde{\nabla}g\ast\tilde{\nabla}g+\Theta\nonumber
\end{equation}
where
\begin{equation}
\begin{split}
\Theta=&-2\tilde{R}_{ikjl}u^{kl}-g^{kl}g_{ip}\tilde{g}^{pq}\tilde{R}_{jkql}-g^{kl}g_{jp}\tilde{g}^{pq}\tilde{R}_{ikql}+\tilde{R}_{ik}u_{j}^{k}+\tilde{R}_{jk}u_{i}^{k}-2\left(n-1\right)\tilde{g}_{ij}\\
=&-4\tilde{R}_{ikjl}u^{kl}+\left(\tilde{g}^{kl}-g^{kl}\right)u_{ip}\tilde{g}^{pq}\tilde{R}_{jkql}+\left(\tilde{g}^{kl}-g^{kl}\right)u_{jp}\tilde{g}^{pq}\tilde{R}_{ikql}+2\left(g_{ab}\tilde{g}^{ak}\tilde{g}^{bl}-g^{kl}\right)\tilde{R}_{ikjl}\\
=&-4\tilde{W}_{ikjl}u^{kl}+4\left(\tilde{g}_{ij}\tilde{g}_{kl}-\tilde{g}_{il}\tilde{g}_{kj}\right)u^{kl}+\left(\tilde{g}^{kl}-g^{kl}\right)u_{ip}\tilde{g}^{pq}\tilde{R}_{jkql}+\left(\tilde{g}^{kl}-g^{kl}\right)u_{jp}\tilde{g}^{pq}\tilde{R}_{ikql}\\
&+2\left(g_{ab}\tilde{g}^{ak}\tilde{g}^{bl}-g^{kl}\right)\tilde{W}_{ikjl}-2\left(g_{ab}\tilde{g}^{ak}\tilde{g}^{bl}-g^{kl}\right)\left(\tilde{g}_{ij}\tilde{g}_{kl}-\tilde{g}_{il}\tilde{g}_{kj}\right)\\
=&-4\tilde{W}_{ikjl}u^{kl}+2\left(g_{ab}\tilde{g}^{ak}\tilde{g}^{bl}-g^{kl}\right)\tilde{W}_{ikjl}+2\left(g_{ab}\tilde{g}^{ak}\tilde{g}^{bl}+g^{kl}-2\tilde{g}^{kl}\right)\left(\tilde{g}_{ij}\tilde{g}_{kl}-\tilde{g}_{il}\tilde{g}_{kj}\right)\\
&+\left(\tilde{g}^{kl}-g^{kl}\right)u_{ip}\tilde{g}^{pq}\tilde{R}_{jkql}+\left(\tilde{g}^{kl}-g^{kl}\right)u_{jp}\tilde{g}^{pq}\tilde{R}_{ikql}.\\
\end{split}\nonumber
\end{equation}

We decompose $\Theta$ into three parts
\begin{equation}
\Theta=\mathbf{P}+\mathbf{Q}+\mathbf{S}\nonumber
\end{equation}
where
\begin{equation}
\begin{split}
\mathbf{P}=&-4\tilde{W}_{ikjl}u^{kl}+2\left(g_{ab}\tilde{g}^{ak}\tilde{g}^{bl}-g^{kl}\right)\tilde{W}_{ikjl}\\
=&-2\left(g_{ab}\tilde{g}^{ak}\tilde{g}^{bl}+g^{kl}-2\tilde{g}^{kl}\right)\tilde{W}_{ikjl},\nonumber
\end{split}
\end{equation}

\begin{equation}
\begin{split}
\mathbf{Q}=&2\left(g_{ab}\tilde{g}^{ak}\tilde{g}^{bl}+g^{kl}-2\tilde{g}^{kl}\right)\left(\tilde{g}_{ij}\tilde{g}_{kl}-\tilde{g}_{il}\tilde{g}_{kj}\right)\\
=&2\left(g_{ab}\tilde{g}^{ab}\tilde{g}_{ij}-g_{ij}+g^{kl}\tilde{g}_{kl}\tilde{g}_{ij}-g^{kl}\tilde{g}_{il}\tilde{g}_{kj}-2(n-1)\tilde{g}_{ij}\right)
\end{split}\nonumber
\end{equation}
and
\begin{equation}
\mathbf{S}=\left(\tilde{g}^{kl}-g^{kl}\right)u_{ip}\tilde{g}^{pq}\tilde{R}_{jkql}+\left(\tilde{g}^{kl}-g^{kl}\right)u_{jp}\tilde{g}^{pq}\tilde{R}_{ikql}.\nonumber
\end{equation}

Choose a coordinate system $\{x^i\}$ such that at one point, we have $\tilde{g}_{ij}=\delta_{ij}$ and $g_{ij}=\lambda_i\delta_{ij}$ with $|\lambda_i-1|\leq2\epsilon$. From the assumption,
\begin{equation*}
\begin{split}
\frac{\p}{\p t}|u|^2=&2\sum_iu_{ii}\frac{\p}{\p t}u_{ii}\\
=&-2\langle\tilde{L}u,u\rangle+2\langle\left(g^{ab}-\tilde{g}^{ab}\right)\tilde{\nabla}_{a}\tilde{\nabla}_{b}u,u\rangle\\
&+\sum_i\left(\tilde{\nabla}g\ast\tilde{\nabla}g\right)_{ii}u_{ii}+2\sum_i\Theta_{ii}u_{ii}.\\
\end{split}
\end{equation*}
We have
\begin{equation*}
\sum_i\left(\tilde{\nabla}g\ast\tilde{\nabla}g\right)_{ii}u_{ii}\leq2\epsilon |\tilde{\nabla}u|^2
\end{equation*}
and
\begin{equation*}
\begin{split}
&2\langle\left(g^{ab}-\tilde{g}^{ab}\right)\tilde{\nabla}_{a}\tilde{\nabla}_{b}u,u\rangle\\
=&\left(g^{ab}-\tilde{g}^{ab}\right)\tilde{\nabla}_{a}\tilde{\nabla}_{b}|u|^2-2\left(g^{ab}-\tilde{g}^{ab}\right)\langle\tilde{\nabla}_{a}u,\tilde{\nabla}_{b}u\rangle\\
\leq&\left(g^{ab}-\tilde{g}^{ab}\right)\tilde{\nabla}_{a}\tilde{\nabla}_{b}|u|^2+4\epsilon|\tilde{\nabla}u|^2.\\
\end{split}
\end{equation*}

Next we only need to check the term $\sum_i\left(\mathbf{P}+\mathbf{Q}+\mathbf{S}\right)_{ii}u_{ii}.$ It is obvious to show that
\begin{equation*}
\begin{split}
\sum_i\mathbf{P}_{ii}u_{ii}=&-2\sum_i\sum_k(\frac{1}{\lambda_k}+\lambda_k-2)(\lambda_i-1)\tilde{W}_{ikik}\\
=&-2\sum_i\sum_k\frac{(\lambda_k-1)^2}{\lambda_k}(\lambda_i-1)\tilde{W}_{ikik}\\
\leq &\epsilon_6 |u|^2,
\end{split}
\end{equation*}
\begin{equation}
\begin{split}
\sum_i\mathbf{Q}_{ii}u_{ii}=&2\sum_i\left(\sum_k\lambda_k-\lambda_i+\sum_k\frac{1}{\lambda_k}-\frac{1}{\lambda_i}-2(n-1)\right)(\lambda_i-1)\\
=&2\sum_i(\lambda_i-1)\sum_{k\neq i}\frac{(\lambda_k-1)^2}{\lambda_k}\\
\leq&\epsilon_7 |u|^2,\\
\end{split}\nonumber
\end{equation}
and
$$\sum_i\mathbf{S}_{ii}u_{ii}\leq \epsilon_8|u|^2.$$

Hence we finish the proof.
\end{proof}

In \cite{SSS} Theorem 3.1, the authors used the first eigenvalue on hyperbolic domains to get a Lyapunov function and obtain the exponential decay of the $L^2$ norm of $g\left(t\right)-\tilde{g}$ with respect to time and finally the convergence of NRDF. Thanks to the strictly stable condition on the background conformally compact Einstein metric $\tilde{g},$ which makes us be able to establish the exponential decay of the $L^2$ norm of $g\left(t\right)-\tilde{g}$ with respect to time and finally get a linearized stability of $\tilde{g}.$

\begin{lem}\label{ltwotimedecay}
Let $g(t)$, $t\in[0,T)$ be a solution to the NRDF (\ref{RicciDeTurckflow}). Then there exists some constant $\kappa>0,$ depending only on $\epsilon,$ $n$ and $\tilde{g}$ such that for all $t\in[0,T),$
\begin{equation*}
\|u\|_{L^{2}\left(M,\tilde{g}\right)}\leq e^{-\kappa t}K.
\end{equation*}
In addition, $\kappa$ is very close to $\lambda.$
\end{lem}
\begin{proof}
We construct a function $\eta=e^{-bs}$ where $b>n-1$ is a constant and $s$ is the distance function to certain fixed large essential set $\Omega\subset M^n$ with respect to $\tilde{g}.$ Hence $\eta\in[0,1]$ and $\eta\equiv1$ on $\Omega.$ And we have that
\begin{equation}
|\tilde{\Delta}\eta|=|\left(b^2|\tilde{\nabla}s|^2-b\tilde{\Delta}s\right)\eta|\leq C\left(n\right)\eta\nonumber
\end{equation}
because $\tilde{\Delta}s=n-1+O\left(e^{-cs}\right)$ for some $c>0$ due to Lemma 2.1 in \cite{HQS}. By the estimates (\ref{smalltimemetricdiff}) and (\ref{smalltimeonederivative}), $\|\eta u\|_{L^{2}\left(M,\tilde{g}\right)}$ and $\|\tilde{\nabla}\left(\eta u\right)\|_{L^{2}\left(M,\tilde{g}\right)}$ are bounded for all $t\in(0,T).$ Then from the assumption that $\tilde{g}$ is a strictly stable conformally compact Einstein metric, we have that
\begin{equation}
\int_{M}\langle \tilde{L}\left(\eta u\right),\eta u\rangle d\mu\geq\lambda\int_{M}|\eta u|^2d\mu.\nonumber
\end{equation}
It follows from Lemma \ref{nondegenerateequ} that
\begin{equation*}
\begin{split}
&\frac{\p}{\p t}\int_{M} |\eta u|^2d\mu\\
=&\int_{M}\eta^2\frac{\p}{\p t}|u|^2d\mu\\
\leq& \int_{M}-2\eta^2\langle\tilde{L}u,u\rangle+\eta^2\left(g^{ab}-\tilde{g}^{ab}\right)\tilde{\nabla}_{a}\tilde{\nabla}_{b}|u|^2+\epsilon_5\eta^2\left(|\tilde{\nabla}u|^2+|u|^2\right)d\mu\\
=&-2\int_{M}\langle\tilde{L}\left(\eta u\right),\eta u\rangle d\mu-2\int_{M}\langle\tilde{\Delta}\eta\cdot u+2\tilde{\nabla}_{\tilde{\nabla}\eta}u,\eta u\rangle d\mu\\
&+\int_{M}\eta^2\left(g^{ab}-\tilde{g}^{ab}\right)\tilde{\nabla}_{a}\tilde{\nabla}_{b}|u|^2d\mu+\epsilon_5\int_{M}\eta^2\left(|\tilde{\nabla}u|^2+|u|^2\right)d\mu\\
=&\mathbf{E}+\mathbf{F}+\mathbf{G}+\mathbf{H}.
\end{split}
\end{equation*}
By direct computation, we have
\begin{equation*}
\begin{split}
\mathbf{E}=&-2\int_{M}\langle\tilde{L}\left(\eta u\right),\eta u\rangle d\mu\\
=&-\left(2-a\right)\int_{M}\langle\tilde{L}\left(\eta u\right),\eta u\rangle d\mu-a\int_{M}\langle\tilde{L}\left(\eta u\right),\eta u\rangle d\mu\\
\leq&-\left(2-a\right)\lambda\int_{M}|\eta u|^2d\mu+a\int_{M}\langle\tilde{\Delta}\left(\eta u\right),\eta u\rangle d\mu+2a\int_{M}\tilde{R}_{ipjq}\eta u^{pq}\eta u^{ij}d\mu\\
\leq&\left(-\left(2-a\right)\lambda+2a\|\tilde{R}m\|_{C^{0}\left(M,\tilde{g}\right)}\right)\int_{M}|\eta u|^2d\mu-a\int_{M}|\tilde{\nabla}\left(\eta u\right)|^2d\mu\\
\end{split}
\end{equation*}
where $a\in(0,2)$ will be determined later and we have used the strictly stable assumption and divergence theorem in the above inequality. The boundary integral that appears in the divergence theorem vanishes because $\eta$ has enough decay at spatial infinity.

\begin{equation*}
\begin{split}
\mathbf{G}=&\int_{M}\eta^2\left(g^{ab}-\tilde{g}^{ab}\right)\tilde{\nabla}_{a}\tilde{\nabla}_{b}|u|^2d\mu\\
=&-\int_{M}\eta^2\tilde{\nabla}_{a}\left(g^{ab}-\tilde{g}^{ab}\right)\tilde{\nabla}_{b}|u|^2d\mu-\int_{M}2\eta\tilde{\nabla}_{a}\eta\cdot\left(g^{ab}-\tilde{g}^{ab}\right)\tilde{\nabla}_{b}|u|^2d\mu\\
\leq&C_1\int_{M}\eta^2|u|\cdot|\tilde{\nabla}u|\cdot|\tilde{\nabla}|u||d\mu+C_1\int_{M}\eta^2|u|^2\cdot|\tilde{\nabla}|u||d\mu\\
\leq&2C_1\epsilon\int_{M}|\eta\tilde{\nabla}u|^2d\mu+C_1\epsilon\int_{M}|\eta u|^2+|\eta\tilde{\nabla}u|^2d\mu\\
\leq&3C_1\epsilon\int_{M}\left(2|\tilde{\nabla}\left(\eta u\right)|^2+2b^2|\eta u|^2\right)d\mu+C_1\epsilon\int_{M}|\eta u|^2d\mu\\
\leq&\left(6b^2+1\right)C_1\epsilon\int_{M}|\eta u|^2d\mu+6C_1\epsilon\int_{M}|\tilde{\nabla}\left(\eta u\right)|^2d\mu\\
\end{split}
\end{equation*}
where we have used the divergence theorem again and Kato's inequality in the above inequality.
\begin{equation*}
\begin{split}
\mathbf{H}=&\int_{M}\epsilon_5\eta^2\left(|\tilde{\nabla}u|^2+|u|^2\right)d\mu\\
\leq&\epsilon_5\int_{M}\left(2|\tilde{\nabla}\left(\eta u\right)|^2+2b^2|\eta u|^2\right)d\mu+\epsilon_5\int_{M}|\eta u|^2d\mu\\
\leq&2\epsilon_5\int_{M}|\tilde{\nabla}\left(\eta u\right)|^2d\mu+\left(2b^2+1\right)\epsilon_5\int_{M}|\eta u|^2d\mu.\\
\end{split}
\end{equation*}
For the last term $\mathbf{F},$ a rough estimate shows that
\begin{equation*}
\begin{split}
\mathbf{F}=&-2\int_{M}\langle\tilde{\Delta}\eta\cdot u+2\tilde{\nabla}_{\tilde{\nabla}\eta}u,\eta u\rangle d\mu\\
\leq&2C\left(n\right)\int_{M}|\eta u|^2d\mu+4b\int_{M}|\eta\tilde{\nabla}u||\eta u| d\mu\\
\leq&2C\left(n\right)\int_{M}|\eta u|^2d\mu+2b\int_{M}\sigma^2|\eta\tilde{\nabla}u|^2+\frac{1}{\sigma^2}|\eta u|^2d\mu\\
\leq&\left(2C\left(n\right)+\frac{2b}{\sigma^2}\right)\int_{M}|\eta u|^2d\mu+2b\sigma^2\int_{M}\left(2|\tilde{\nabla}\left(\eta u\right)|^2+2b^2|\eta u|^2\right)d\mu\\
\leq&\left(2C\left(n\right)+\frac{2b}{\sigma^2}+4b^3\sigma^2\right)\int_{M}|\eta u|^2d\mu+4b\sigma^2\int_{M}|\tilde{\nabla}\left(\eta u\right)|^2d\mu\\
\end{split}
\end{equation*}
where $\sigma$ is a sufficiently small constant which will be determined later.

Combine all these estimates together, we get
\begin{equation}\label{equforetau}
\frac{\p}{\p t}\int_{M} |\eta u|^2d\mu\leq C\left(n,\epsilon,\tilde{g},\sigma,a\right)\int_{M}|\eta u|^2d\mu-A\int_{M}|\tilde{\nabla}\left(\eta u\right)|^2d\mu
\end{equation}
where $-A=-a+6C_1\epsilon+2\epsilon_5+4b\sigma^2.$ For any fixed $a\in(0,2),$ we can always choose $\epsilon$ and $\sigma$ to be sufficiently small such that $A>0.$ Hence the above equation (\ref{equforetau}) becomes an ODE
\begin{equation}
\frac{\p}{\p t}\int_{M} |\eta u|^2d\mu\leq C_2\int_{M}|\eta u|^2d\mu.\nonumber
\end{equation}
By solving it, we get
\begin{equation}
\int_{M} |\eta u\left(t\right)|^2d\mu\leq e^{C_2t}\int_{M} |\eta u\left(0\right)|^2d\mu\nonumber
\end{equation}
for all $t\in[0,T).$
Let $\Omega$ exhaust to $M,$ which implies that for all $t\in[0,T),$
\begin{equation}\label{L2boundofu}
\|u\left(t\right)\|_{L^{2}\left(M,\tilde{g}\right)}\leq e^{\frac{C_2}{2}T}K.
\end{equation}
Also by the equation (\ref{equforetau}), we have
\begin{equation}
Ae^{-C_2t}\int_{M}|\tilde{\nabla}\left(\eta u\right)|^2d\mu\leq-\frac{\p}{\p t}\left(e^{-C_2t}\int_{M} |\eta u|^2d\mu\right),\nonumber
\end{equation}
which implies that for all $t\in(0,T),$
\begin{equation}
\int_{0}^{t}\int_{M}|\tilde{\nabla}\left(\eta u\right)|^2d\mu ds\leq\frac{K^2}{A}e^{C_2T}.\nonumber
\end{equation}
Furthermore, let $\Omega$ exhaust to $M,$ we get that for all $t\in(0,T),$
\begin{equation}\label{L2boundofnablau}
\int_{0}^{t}\int_{M}|\tilde{\nabla}u|^2d\mu ds\leq\frac{K^2}{A}e^{C_2T}.
\end{equation}
Now we take a deep investigation into the term $\mathbf{F}.$ Since $\eta\equiv1$ on $\Omega,$ we see that

\begin{equation*}
\begin{split}
\mathbf{F}=&-2\int_{M}\langle\tilde{\Delta}\eta\cdot u+2\tilde{\nabla}_{\tilde{\nabla}\eta}u,\eta u\rangle d\mu\\
=&-2\int_{M\backslash\Omega}\eta\tilde{\Delta}\eta|u|^2d\mu-4\int_{M\backslash\Omega}\langle\tilde{\nabla}_{\tilde{\nabla}\eta}u,\eta u\rangle d\mu.\\
\end{split}
\end{equation*}
Hence
\begin{equation*}
\begin{split}
&\frac{\p}{\p t}\int_{M}|\eta u|^2d\mu\\
\leq& \left(-\left(2-a\right)\lambda+2a\|\tilde{R}m\|_{C^{0}\left(M,\tilde{g}\right)}+\left(6b^2+1\right)C_1\epsilon+\left(2b^2+1\right)\epsilon_5\right)\int_{M}|\eta u|^2d\mu\\
&+\left(-a+6C_1\epsilon+2\epsilon_5\right)\int_{M}|\tilde{\nabla}\left(\eta u\right)|^2d\mu\\
&-2\int_{M\backslash\Omega}\eta\tilde{\Delta}\eta|u|^2d\mu-4\int_{M\backslash\Omega}\langle\tilde{\nabla}_{\tilde{\nabla}\eta}u,\eta u\rangle d\mu.\\
\end{split}
\end{equation*}
This time we still choose $a$ to be sufficiently small. For this fixed $a\in\left(0,2\right),$ we see that we can pick an $\epsilon$ sufficiently small such that
$$-2\kappa=-\left(2-a\right)\lambda+2a\|\tilde{R}m\|_{C^{0}\left(M,\tilde{g}\right)}+\left(6b^2+1\right)C_1\epsilon+\left(2b^2+1\right)\epsilon_5<0$$ and $$-a+6C_1\epsilon+2\epsilon_5<0.$$
By the fact that $\epsilon_5$ tends to zero as $\epsilon$ goes to zero, $\kappa$ is indeed very close to $\lambda.$ It follows that
\begin{equation}
\frac{\p}{\p t}\int_{M}|\eta u|^2d\mu\leq -2\kappa\int_{M}|\eta u|^2d\mu-2\int_{M\backslash\Omega}\eta\tilde{\Delta}\eta|u|^2d\mu-4\int_{M\backslash\Omega}\langle\tilde{\nabla}_{\tilde{\nabla}\eta}u,\eta u\rangle d\mu,\nonumber
\end{equation}
which implies that
\begin{equation}
e^{2\kappa t}\int_{M}|\eta u\left(t\right)|^2d\mu\leq \int_{M}|\eta u\left(0\right)|^2d\mu+|e^{2\kappa t}\int_{0}^{t}\int_{M\backslash\Omega}-2\eta\tilde{\Delta}\eta|u|^2-4\langle\tilde{\nabla}_{\tilde{\nabla}\eta}u,\eta u\rangle d\mu ds|.\nonumber
\end{equation}
However,
\begin{equation}\label{remainder}
\begin{split}
&|e^{2\kappa t}\int_{0}^{t}\int_{M\backslash\Omega}-2\eta\tilde{\Delta}\eta|u|^2-4\langle\tilde{\nabla}_{\tilde{\nabla}\eta}u,\eta u\rangle d\mu ds|\\
\leq&2C\left(n\right)e^{2\kappa T}\int_{0}^{t}\int_{M\backslash\Omega}|\eta u|^2d\mu ds+2be^{2\kappa T}\int_{0}^{t}\int_{M\backslash\Omega}|\eta\tilde{\nabla}u|^2+|\eta u|^2d\mu ds\\
\leq&C\left(n,T\right)\left(\int_{0}^{t}\int_{M\backslash\Omega}|u|^2d\mu ds+\int_{0}^{t}\int_{M\backslash\Omega}|\tilde{\nabla}u|^2d\mu ds\right).\\
\end{split}
\end{equation}
By means of (\ref{L2boundofu}) and (\ref{L2boundofnablau}), it immediately follows that the right hand side of (\ref{remainder}) goes to zero as $\Omega$ exhausts to the whole $M.$
Therefor for all $t\in[0,T),$  we obtain that
\begin{equation}
e^{2\kappa t}\int_{M}|u\left(t\right)|^2d\mu\leq \int_{M}|u\left(0\right)|^2d\mu.\nonumber
\end{equation}
Thus we finish to prove this lemma.
\end{proof}

Now we move to the long-time existence and convergence of NRDF.
\begin{proof}[Proof of Theorem \ref{stability}]
Now that we have the exponential decay of the $L^2$-norm of $u$, by the same idea as that in Theorem 3.2, Theorem 3.3 and Theorem 3.4 in \cite{SSS}, we will get exponential decay in the $C^0$-norm of $u$ and finally the long-time existence and convergence. But here we have to check that the injective radius is bounded along NRDF in order that the interior estimates and gradient estimates make sense. WLOG, assume that $T>1.$ As to the injective radius, we only need to check that it has a uniformly lower bound for all $t\in[1,T).$ In fact, the estimates $\|u\|_{C^0\left(M^n,\tilde{g}\right)}\leq2\epsilon,$ together with the fact that $\tilde{g}$ is $C^{2,\alpha}$ comformally compact, implies that there exists some $v=v\left(\tilde{g},\epsilon,n\right)>0$ such that $vol(B_{g\left(t\right)}(p,1))\geq v$ for all $p\in M.$ On the other hand, it follows from the estimates (\ref{smalltimemetricdiff}), (\ref{smalltimeonederivative}) and (\ref{smalltimetwoderivative}) that the sectional curvature of $g\left(t\right)$ has an upper bound for $t\in[1,T).$ Hence the injective radius is bounded uniformly along NRDF for all $t\in[1,T).$ This uniform bound depends only on $\epsilon,$ $n$ and $\tilde{g}.$ Then $|\tilde{\nabla}^ku\left(t\right)|,$ $k\in\mathbb{N},$ decay exponentially by the same arguments in Theorem 3.3 and Theorem 3.4 in \cite{SSS}, therefore the injective radius is bounded along NRDF until time goes to infinity.
\end{proof}

\section{Proof of the Volume Comparison}
As an application of the long-time existence and convergence of the NRDF, we prove Theorem \ref{main1} and Theorem \ref{rigidity}. Now Let $\tilde{g}$ and $g$ be as stated in Theorem \ref{main1} that $\left(M^n, \tilde{g}\right)$ is a $C^{2,\alpha}$ strictly stable conformally compact Einstein manifold and $g$ is another complete noncompact Riemannian metric on $M^n.$ For all $n\geq4$ and $\tau>n-1,$ assume that
$$\|e^{\tau\rho}\left(g-\tilde{g}\right)\|_{C^1\left(M^n,\tilde{g}\right)}\leq \epsilon,$$
and
$$R\left(g\right)\geq-n\left(n-1\right).$$
In order to control the decay rate both in time and at spatial infinity uniformly, we choose $\delta\in(0,\tau]$ and $\gamma$ satisfying
$$\delta\in\left(n-1,\frac{\left(n-1\right)+\sqrt{\left(n-1\right)^2+4\left(n-1\right)}}{2}\right)$$ and $$\gamma\in\left(\frac{n-1}{2}-\sqrt{\frac{\left(n-1\right)^2}{4}-2},\frac{n-1}{2}+\sqrt{\frac{\left(n-1\right)^2}{4}-2}\right).$$ We remark that $\gamma$ makes sense only if $n\geq4.$ Obviously, $$\|e^{\delta\rho}\left(g-\tilde{g}\right)\|_{C^1\left(M^n,\tilde{g}\right)}\leq \epsilon.$$ It follows from Theorem \ref{stability} that NRDF $g\left(t\right)$ starting from $g$ exists globally and converges exponentially to $\tilde{g}.$
And there exists some $\tilde{\epsilon}>0$ which depends only on $\epsilon$, $n$ and $\tilde{g}$ such that
\begin{equation}\label{gclosetogtilde}
\|g\left(t\right)-\tilde{g}\|_{C^0\left(M^n,\tilde{g}\right)}\leq\tilde{\epsilon}
\end{equation}
for all $t\in[0,\infty).$

However, as to the estimates on the derivatives, it is different for the time interval $t\in[0,1]$ and $t\in[1,\infty).$ For all $t\in[0,1],$ the estimates on the derivatives were improved in \cite{Sim} Lemma 2.1, that is, there exist some constants $\epsilon_9$ and $\epsilon_{10}$ which only depend on $\epsilon,$ $n$ and $\tilde{g}$ such that
\begin{equation}\label{smalltimeonederivative1}
\|\tilde{\nabla}g(t)\|_{C^0\left(M^n,\tilde{g}\right)}\leq \epsilon_9,
\end{equation}
and
\begin{equation}\label{smalltimetwoderivative2}
\|\tilde{\nabla}^{2}g(t)\|_{C^0\left(M^n,\tilde{g}\right)}\leq\frac{\epsilon_{10}}{\sqrt{t}}.
\end{equation}
For all $t\in[1,\infty),$ it follows immediately from the long-time existence and convergence of the NRDF that there exist some constants $\sigma_{i}=\sigma_{i}\left(n,\epsilon,\tilde{g}\right)>0$ and $C_{i+2}>0$ such that
\begin{equation}\label{bigtimeonederivative}
\|\tilde{\nabla}^{i}g\left(t\right)\|_{C^0\left(M^n,\tilde{g}\right)}\leq C_{i+2} e^{-\sigma_{i} t},~~~~\text{ }~~~~i=1,2,3.
\end{equation}

As a consequence of Lemma \ref{evolutionequ}, we show the following:

\begin{lem}\label{decayingestimate}
Let $g(t)$, $t\in[0,\infty)$ be a solution to the NRDF (\ref{RicciDeTurckflow}). Then there exist constants $\hat{\epsilon}, \bar{\epsilon}>0$, depending only on $\epsilon$, $n$ and $\tilde{g}$ such that
\begin{enumerate}
  \item for all $t\in[0,1],$ $|g\left(t\right)-\tilde{g}|\leq \hat{\epsilon} e^{-\delta\rho},$ $|\tilde{\nabla}g\left(t\right)|\leq \hat{\epsilon} e^{-\delta\rho},$ $|\tilde{\nabla}^2g|\leq\frac{\hat{\epsilon}}{\sqrt{t}}e^{-\delta\rho};$
  \item for all $t\in[1,\infty),$ $|g\left(t\right)-\tilde{g}|\leq \bar{\epsilon} e^{-\sigma_4 t}e^{-\gamma\rho},$ $|\tilde{\nabla}^{k}g\left(t\right)|\leq \bar{\epsilon} e^{-\sigma_4 t}e^{-\gamma\rho},$ $k=1,2,$ where $\sigma_4>0$ is an arbitrarily small constant.
\end{enumerate}
\end{lem}

\begin{proof}
For $t\in[0,1],$ let $\phi=e^{\mu\rho}\left(|g-\tilde{g}|^2+|\tilde{\nabla}g|^2+t|\tilde{\nabla}^2g|^2\right)$. It follows from Lemma \ref{evolutionequ} and the estimates (\ref{gclosetogtilde}), (\ref{smalltimeonederivative1}) and (\ref{smalltimetwoderivative2}) that
\begin{equation}
\begin{split}
\frac{\p}{\p t}\phi&\leq g^{ij}\tilde{\nabla}_{i}\tilde{\nabla}_{j}\phi+ \left(C_6-\mu^2 g^{ij} \tilde{\nabla}_{i}\rho\cdot\tilde{\nabla}_{j}\rho-\mu g^{ij}\tilde{\nabla}_{i}\tilde{\nabla}_{j}\rho\right)\phi\\
&-2\mu e^{\mu\rho}g^{ij} \tilde{\nabla}_{i}\rho\cdot\tilde{\nabla}_{j}\left(|g-\tilde{g}|^2+|\tilde{\nabla}g|^2+t|\tilde{\nabla}^2g|^2\right)\\
&-\left(2-\epsilon_{11}\right)e^{\mu\rho}\left(|\tilde{\nabla}g|^2+|\tilde{\nabla}^{2}g|^2+t|\tilde{\nabla}^{3}g|^2\right)+\left(1+\epsilon_{12}\sqrt{t}\right)|\tilde{\nabla}^2g|^2\\
\end{split}\nonumber
\end{equation}
In order to get the spatial decay of the metric, we modify the essential set to be sufficiently large, together with that $\tilde{g}$ is $C^{2,\alpha} $ conformally compact,  by Lemma 2.1 in \cite{HQS}, we have that $$\tilde{\Delta}\rho=n-1+O\left(e^{-a\rho}\right)$$ for some $a>0.$ In the light of the estimates
$$g^{ij}\tilde{\nabla}_{i}\rho\tilde{\nabla}_{j}\rho\geq\left(1-\tilde{\epsilon}\right)|\tilde{\nabla}\rho|^{2}=1-2\epsilon,$$
$$g^{ij}\tilde{\nabla}_{i}\tilde{\nabla}_{j}\rho\geq\left(1-\tilde{\epsilon}\right)\tilde{\Delta}\rho=\left(1-\epsilon_{13}\right)\left(n-1\right),$$
and
\begin{equation}
\begin{split}
&|g^{ij}\tilde{\nabla}_{i}\rho\tilde{\nabla}_{j}|\tilde{\nabla}^{k-1}\left(g-\tilde{g}\right)|^2|\\
\leq  &\left(1+\tilde{\epsilon}\right)|\tilde{\nabla}\langle\tilde{\nabla}^{\left(k-1\right)}\left(g-\tilde{g}\right),\tilde{\nabla}^{\left(k-1\right)}\left(g-\tilde{g}\right)\rangle_{\tilde{g}}|\\
\leq &\left(1+\tilde{\epsilon}\right)\left(b^2|\tilde{\nabla}^{k}\left(g-\tilde{g}\right)|^{2}+\frac{1}{b^2}|\tilde{\nabla}^{k-1}\left(g-\tilde{g}\right)|^{2}\right),\\
\end{split}\nonumber
\end{equation}
where $b>0$ is arbitrary, we can choose appropriate $b$ and get that there exists a constant $C_7>0,$ depending only on $\epsilon$, $\mu,$ $n$ and $\tilde{g}$, such that
\begin{equation}
\frac{\p}{\p t}\phi\leq g^{ij}\tilde{\nabla}_{i}\tilde{\nabla}_{j}\phi+ C_7\phi\nonumber
\end{equation}
Let $\mu=2\delta$, by maximum principle due to Karp and Li (See Theorem 7.39 in \cite{CLN}), we get for $t\in[0,1],$ $$\phi\leq\hat{\epsilon}^2:=\phi(\cdot,0)e^{C_7}.$$
Hence $$|\tilde{\nabla}^2g|\leq\frac{\hat{\epsilon}}{\sqrt{t}}e^{-\delta\rho},$$ for all $t\in[0,1].$

In particular, at $t=1,$
\begin{equation}\label{timeoneestimate}
|g-\tilde{g}|\left(\cdot,1\right), |\tilde{\nabla}g|\left(\cdot,1\right), |\tilde{\nabla}^2g|\left(\cdot,1\right)\leq\hat{\epsilon}e^{-\delta\rho}.
\end{equation}

For $t\in[1,\infty),$ let $\varphi=e^{\sigma_4 t}e^{\nu\rho}\left(|g-\tilde{g}|^2+a|\tilde{\nabla}g|^2+b|\tilde{\nabla}^{2}g|^2\right).$ Due to (\ref{gclosetogtilde}) and (\ref{bigtimeonederivative}), we can choose $a$ and $b$ such that
\begin{equation}
\begin{split}
\frac{\p}{\p t}\varphi&\leq g^{ij}\tilde{\nabla}_{i}\tilde{\nabla}_{j}\varphi+ \left(4+4|\tilde{W}|+\tilde{\epsilon}+\sigma_4-\nu^2 g^{ij} \tilde{\nabla}_{i}\rho\cdot\tilde{\nabla}_{j}\rho-\nu g^{ij}\tilde{\nabla}_{i}\tilde{\nabla}_{j}\rho\right)\varphi\\
&-\left(2-\epsilon_{14}\right)e^{\sigma_4 t}e^{\nu\rho}\left(|\tilde{\nabla}g|^2+a|\tilde{\nabla}^{2}g|^2+b|\tilde{\nabla}^{3}g|^2\right)\\
&-2\nu e^{\sigma_4 t}e^{\nu\rho}g^{ij} \tilde{\nabla}_{i}\rho\cdot\tilde{\nabla}_{j}\left(|g-\tilde{g}|^2+a|\tilde{\nabla}g|^2+b|\tilde{\nabla}^{2}g|^2\right)\\
\end{split}\nonumber
\end{equation}
Let $c+d=-2\nu,$ then
\begin{equation}
\begin{split}
&-2\nu e^{\nu\rho}g^{ij}\tilde{\nabla}_{i}\rho\cdot\tilde{\nabla}_{j}|\tilde{\nabla}^{k}\left(g-\tilde{g}\right)|^2\\
=&\left(c+d\right)e^{\nu\rho}g^{ij}\tilde{\nabla}_{i}\rho\cdot\tilde{\nabla}_{j}|\tilde{\nabla}^{k}\left(g-\tilde{g}\right)|^2\\
=&ce^{\nu\rho}g^{ij}\tilde{\nabla}_{i}\rho\cdot\tilde{\nabla}_{j}|\tilde{\nabla}^{k}\left(g-\tilde{g}\right)|^2+de^{\nu\rho}g^{ij}\tilde{\nabla}_{i}\rho\cdot\tilde{\nabla}_{j}|\tilde{\nabla}^{k}\left(g-\tilde{g}\right)|^2\\
\leq&|c|\left(1+\tilde{\epsilon}\right)\left(m^2e^{\nu\rho}|\tilde{\nabla}^{k+1}\left(g-\tilde{g}\right)|^2+\frac{1}{m^2}e^{\nu\rho}|\tilde{\nabla}^{k}\left(g-\tilde{g}\right)|^2\right)\\
&+\left(-b\nu g^{ij}\tilde{\nabla}_{i}\rho\cdot\tilde{\nabla}_{j}\rho\left(e^{\nu\rho}|\tilde{\nabla}^{k}\left(g-\tilde{g}\right)|^2\right)+dg^{ij}\tilde{\nabla}_{i}\rho\cdot\tilde{\nabla}_{j}\left(e^{\nu\rho}|\tilde{\nabla}^{k}\left(g-\tilde{g}\right)|^2\right)\right)\\
=&|c|\left(1+\tilde{\epsilon}\right)m^2e^{\nu\rho}|\tilde{\nabla}^{k+1}\left(g-\tilde{g}\right)|^2+dg^{ij}\tilde{\nabla}_{i}\rho\cdot\tilde{\nabla}_{j}\left(e^{\nu\rho}|\tilde{\nabla}^{k}\left(g-\tilde{g}\right)|^2\right)\\
&+\left(\frac{|c|\left(1+\tilde{\epsilon}\right)}{m^2}-d\nu g^{ij}\tilde{\nabla}_{i}\rho\cdot\tilde{\nabla}_{j}\rho\right)\left(e^{\nu\rho}|\tilde{\nabla}^{k}\left(g-\tilde{g}\right)|^2\right).\\
\end{split}\nonumber
\end{equation}
Let $|c|\left(1+\tilde{\epsilon}\right)m^2=2-\epsilon_{14},$ then we have
\begin{equation}\label{varphi}
\begin{split}
\frac{\p}{\p t}\varphi & \leq g^{ij}\tilde{\nabla}_{i}\tilde{\nabla}_{j}\varphi+\left(-2\nu-c\right)g^{ij}\tilde{\nabla}_{i}\rho\cdot\tilde{\nabla}_{j}\varphi\\
&\left(4+4|\tilde{W}|+\tilde{\epsilon}+\sigma_4+\frac{\left(1+\tilde{\epsilon}\right)^2c^2}{2-\epsilon_{14}}+\left(\nu^2+c\nu\right) g^{ij} \tilde{\nabla}_{i}\rho\cdot\tilde{\nabla}_{j}\rho-\nu g^{ij}\tilde{\nabla}_{i}\tilde{\nabla}_{j}\rho\right)\varphi\\
&\leq g^{ij}\tilde{\nabla}_{i}\tilde{\nabla}_{j}\varphi+\left(-2\nu-c\right)g^{ij}\tilde{\nabla}_{i}\rho\cdot\tilde{\nabla}_{j}\varphi+C_8\varphi\\
\end{split}
\end{equation}
where
\begin{equation}
C_8=4+4l+\tilde{\epsilon}+\sigma_4+\frac{\left(1+\tilde{\epsilon}\right)^2c^2}{2-\epsilon_{14}}+\left(\nu^2+c\nu\right)\left(1+\tilde{\epsilon}\right)-\nu\left(n-1\right)\left(1-\tilde{\epsilon}\right).\nonumber
\end{equation}
For the purpose of getting the decay at spatial infinity, we choose a sufficiently large essential set. Then outside the essential set, we see that $|\tilde{W}|\leq l$ and the constant $l>0$ can be arbitrarily small. As $\tilde{\epsilon}$ and $\sigma_4$ also can be sufficiently small, we choose $c=-\nu$ and $\nu=2\gamma,$ then $C_8\leq0.$
Since we have estimates (\ref{timeoneestimate}) at $t=1,$ $$\varphi(\cdot,1)=e^{\sigma_4}e^{2\gamma\rho}\left(|g-\tilde{g}|^2+a|\tilde{\nabla}g|^2+b|\tilde{\nabla}^{2}g|^2\right)\leq e^{\sigma_4}\hat{\epsilon}^2,$$ it follows from the evolution equation (\ref{varphi}) and maximum principle that
\begin{equation}
\varphi\left(\cdot,t\right)\leq \bar{\epsilon}^2=\frac{e^{\sigma_4}\hat{\epsilon}^2}{\min\{a,b,1\}}\nonumber
\end{equation}
for $t\in[1,\infty).$ Note that for our purpose we need a variant of the maximum principle in Theorem 4.2 in \cite{QSW}, where Qing, Shi and Wu proved a variant of the maximum principle in Theorem 4.3 in \cite{EH}, originally from \cite{LT}. The proof goes the same as that in \cite{EH} Theorem 4.3 and \cite{QSW} Theorem 4.2 if we change the time-dependent laplacian $\Delta_{g\left(t\right)}$ there to be $g^{ij}\tilde{\nabla}_i\tilde{\nabla}_j$ in our case.

\end{proof}

\begin{lem}\label{decayofV}
Under the NRDF, for all $t\in[0,\infty),$ there exists some constant $\epsilon'>0$ depending only on $\epsilon$, $n$ and $\tilde{g},$ such that the $1$-form $V$ satisfies
\begin{equation}
|V\left(t\right)|\leq \epsilon'e^{-\tilde{\sigma} t} e^{-\delta\rho},\nonumber
\end{equation}
where $\tilde{\sigma}>0$ is an arbitrarily small constant.
\end{lem}
\begin{proof}
Let $E=Ric+\left(n-1\right)g,$ we consider the evolution equation of $V$ under NRDF,
\begin{equation}
\begin{split}
\frac{\p}{\p t}V_{j}&=\frac{\p}{\p t}[g_{jk}g^{pq}\left(\Gamma_{pq}^{k}-\tilde{\Gamma}_{pq}^{k}\right)]\\
&=g_{jk}g^{pq}\frac{\p}{\p t}\Gamma_{pq}^{k}+\frac{\p}{\p t}\left(g_{jk}g^{pq}\right)\left(\Gamma_{pq}^{k}-\tilde{\Gamma}_{pq}^{k}\right)\\
&=\mathbf{I}+\mathbf{J}\\
\end{split}\nonumber
\end{equation}
where
\begin{equation}
\begin{split}
\mathbf{I}=&\frac{1}{2}g^{kl}g_{jk}g^{pq}\left(\nabla_{p}\left(-2E_{lq}+\nabla_{l}V_{q}+\nabla_{q}V_{l}\right)+\nabla_{q}\left(-2E_{lp}+\nabla_{l}V_{p}+\nabla_{p}V_{l}\right)\right)\\
&-\frac{1}{2}g^{kl}g_{jk}g^{pq}\left(\nabla_{l}\left(-2E_{pq}+\nabla_{q}V_{p}+\nabla_{p}V_{q}\right)\right)\\
=&-g^{pq}\left(\nabla_{p}R_{jq}+\nabla_{q}R_{jp}-\nabla_{j}R_{pq}\right)+\Delta V_{j}\\
&+\frac{1}{2}g^{pq}\left(\nabla_{p}\nabla_{j}V_{q}-\nabla_{j}\nabla_{p}V_{q}+\nabla_{q}\nabla_{j}V_{p}-\nabla_{j}\nabla_{q}V_{p}\right)\\
=&-\left(2\left(div_{g} Ric\right)_{j}-\nabla_{j}R\right)+\Delta V_{j}+R_{j}^{k}V_{k}\\
=&\Delta V_{j}+R_{j}^{k}V_{k}\\
\end{split}\nonumber
\end{equation}
and
\begin{equation}
\begin{split}
\mathbf{J}=&\frac{\p}{\p t}\left(g_{jk}g^{pq}\right)\left(\Gamma_{pq}^{k}-\tilde{\Gamma}_{pq}^{k}\right)\\
=&\left(\frac{\p}{\p t}g_{jk}\cdot g^{pq}+g_{jk}\cdot\frac{\p}{\p t}g^{pq}\right)\left(\Gamma_{pq}^{k}-\tilde{\Gamma}_{pq}^{k}\right).\\
\end{split}\nonumber
\end{equation}
In light of
\begin{equation}
\begin{split}
&\Gamma_{pq}^{k}-\tilde{\Gamma}_{pq}^{k}\\
=&\frac{1}{2}g^{kl}\left(g_{li,j}+g_{lj,i}-g_{ij,l}\right)-\frac{1}{2}\tilde{g}^{kl}\left(\tilde{g}_{li,j}+\tilde{g}_{lj,i}-\tilde{g}_{ij,l}\right)\\
=&\frac{1}{2}\left(g^{kl}-\tilde{g}^{kl}\right)\left(g_{li,j}+g_{lj,i}-g_{ij,l}\right)-\frac{1}{2}\tilde{g}^{kl}\left(u_{li,j}+u_{lj,i}-u_{ij,l}\right)\\
=&\frac{1}{2}\left(g^{kl}-\tilde{g}^{kl}\right)\left(g_{li,j}+g_{lj,i}-g_{ij,l}\right)-\frac{1}{2}\tilde{g}^{kl}\left(\left(\tilde{\nabla}_jg\right)_{li}+\left(\tilde{\nabla}_ig\right)_{lj}-\left(\tilde{\nabla}_lg\right)_{ij}+2\tilde{\Gamma}_{ij}^au_{al}\right),\\
\end{split}\nonumber
\end{equation}
we have
\begin{equation}
|\Gamma_{pq}^{k}-\tilde{\Gamma}_{pq}^{k}|\left(x\right)\leq C_9\|\tilde{\nabla}g\left(t\right)\|_{C^0\left(M^n,\tilde{g}\right)}\cdot|u\left(t\right)|\left(x\right)+C_9|\tilde{\nabla}g\left(t\right)|\left(x\right).\nonumber
\end{equation}
According to the evolution equation of $u$ in Lemma \ref{evolutionequ}, we have
\begin{equation}
|\frac{\p}{\p t}g_{jk}|\left(x\right)\leq C_{10}\left(|\tilde{\nabla}^2g\left(t\right)|\left(x\right)+|\tilde{\nabla}g\left(t\right)|\left(x\right)+|u\left(t\right)|\left(x\right)\right).\nonumber
\end{equation}
It follows from Lemma \ref{decayingestimate} that for all $t\in[1,\infty),$
\begin{equation}
\mathbf{J}\leq C_{11}\bar{\epsilon}^2e^{-2\sigma_4t} e^{-2\gamma\rho}.\nonumber
\end{equation}
On the other hand, for all $t\in[1,\infty),$ $$|R_{j}^{k}+\left(n-1\right)\delta_{j}^{k}|=|R_{j}^{k}-\tilde{R}_{j}^{k}|\leq C_{12}\|g-\tilde{g}\|_{C^2\left(M^n,\tilde{g}\right)}\leq\epsilon_{15}.$$
Hence we have
\begin{equation}
\frac{\p}{\p t}|V|\leq \Delta |V|-\left(n-1-\epsilon_{15}\right)|V|+C_{11}\bar{\epsilon}^2e^{-2\sigma_4 t}e^{-2\gamma\rho}\nonumber
\end{equation}
where $\Delta$ is with respect to $g.$
At $t=1,$ by the above formula of $\Gamma_{pq}^{k}-\tilde{\Gamma}_{pq}^{k}$ again, we have $$|V\left(\cdot,1\right)|\leq \epsilon_{16} e^{-\delta\rho}.$$
Let $2\gamma=\delta$ and $v=e^{\delta\rho}|V|,$ we can see that $v$ satisfies
\begin{equation}\label{v}
\begin{split}
\frac{\p}{\p t}v\leq&\Delta v-2\delta\nabla\rho\cdot\nabla v-\left(\delta\Delta\rho+n-1-\epsilon_{15}-\delta^2|\nabla\rho|^2\right)v+C_{11}\bar{\epsilon}^2e^{-2\sigma_4 t}\\
\leq&\Delta v-2\delta\nabla\rho\cdot\nabla v-Bv+C_{11}\bar{\epsilon}^2e^{-2\sigma_4 t}\\
\end{split}\nonumber
\end{equation}
where $B=\delta\Delta\rho+n-1-\epsilon_{15}-\delta^2|\nabla\rho|^2$ is a positive constant provided that $\delta$ belongs to certain range as discussed before. Choose $B\neq 2\sigma_4$ and consider ODE
\begin{equation}\label{ODE}
\left\{
  \begin{array}{ll}
    \frac{du}{dt}=-Bu+C_{11}\bar{\epsilon}^2e^{-2\sigma_4 t}, t\in[1,\infty),\\
    u\left(1\right)=\epsilon_{16}.
  \end{array}
\right.
\end{equation}
the solution
\begin{equation}
u\left(t\right)=\epsilon_{16}e^{B}e^{-Bt}+\frac{C_{11}\bar{\epsilon}^2}{B-2\sigma_4}\left(e^{-2\sigma_4 t}-e^{B-2\sigma_4}e^{-Bt}\right).\nonumber
\end{equation}

Since $u$ is a subsolution to the equation (\ref{v}) with $v(\cdot,1)\leq u(1)$, due to Theorem 4.2 in \cite{QSW}, we have
$v\left(\cdot,t\right)\leq u\left(t\right)$ for all $t\in[1,\infty).$ Hence
\begin{equation}
|V\left(t\right)|\leq \epsilon_{17}e^{-\tilde{\sigma} t} e^{-\delta\rho}\nonumber
\end{equation}
for all $t\in[1,\infty),$ where $\tilde{\sigma}=\min\{2\sigma_4, B\}$ and $\epsilon_{17}$ depends only on $\epsilon$, $n$ and $\tilde{g}.$ Together with Lemma \ref{decayingestimate}, we conclude that for all $t\in[0,\infty),$
\begin{equation}
|V\left(t\right)|\leq \epsilon'e^{-\tilde{\sigma} t} e^{-\delta\rho}.\nonumber
\end{equation}

\end{proof}

\begin{lem}\label{decayofscalarcurvature}
Under the NRDF, for all $t\in[0,\infty),$ the scalar curvature $R$ satisfies
$$R\left(t\right)\geq-n\left(n-1\right).$$
Moreover, there exists some constant $\epsilon''>0$ depending only on $\epsilon$, $n$ and $\tilde{g},$ such that
$$|R\left(t\right)+n(n-1)|\leq \left(\frac{\epsilon''}{\sqrt{t}}+\epsilon''\right)e^{-\delta \rho},~~~\text{for~all}~~~t\in[0,1],$$
and
$$|R\left(t\right)+n(n-1)|\leq \epsilon''e^{-\bar{\sigma}t}e^{-\delta \rho},~~~\text{for~all}~~~t\in[1,\infty),$$
where $\bar{\sigma}>0$ is an arbitrarily small constant.
\end{lem}
\begin{proof}

By a direct computation, we see that under the NRF of $\bar{g}$, the scalar curvature $\bar{R}=R(\bar{g}(t))$ satisfies the following evolution equation
\begin{equation}\label{evolutionscalarcurvature}
\frac{\p}{\p t}\bar{R}=\Delta \bar{R}+ 2|\bar{R}ic+(n-1)\bar{g}|_{\bar{g}}^2 -2(n-1)(\bar{R}+n(n-1)),\nonumber
\end{equation}
where $\Delta$ and $|\cdot|_{\bar{g}}$ are with respect to the solution of  NRF $\bar{g}\left(t\right).$
Let $\bar{S}=\bar{R}+n(n-1),$ under $NRF,$ it satisfies
\begin{equation}\label{evolutionscalarcurvature}
\frac{\p}{\p t}\bar{S}=\Delta \bar{S}+ 2|\bar{R}ic+(n-1)\bar{g}|_{\bar{g}}^2 -2(n-1)\bar{S},\nonumber
\end{equation}
with $\bar{S}(0)\geq0$. Then by maximum principle, we see that
$$\bar{S}(t)\geq0,$$
which means
$$
R\left(\bar{g}\left(t\right)\right)\geq -n(n-1).
$$
Because of the diffeomorphism invariance, we have
$$
R\left(g\left(t\right)\right)=R\left(\Phi^{\ast}_{t}g\left(t\right)\right)=R\left(\bar{g}\left(t\right)\right)\geq -n(n-1).
$$
To obtain the decay both in time and at spatial infinity, we rewrite the formula of scalar curvature to be
\begin{equation}
\begin{split}
&R\left(t\right)+n\left(n-1\right)\\
=&R\left(t\right)-\tilde{R}\\
=&\left(g^{ab}-\tilde{g}^{ab}\right)\left(\Gamma_{ab,c}^{c}-\Gamma_{ac,b}^c+\Gamma_{ab}^{d}\Gamma_{cd}^c-\Gamma_{ac}^d\Gamma_{bd}^c\right)\\
&+\tilde{g}^{ab}\left(\Gamma_{ab,c}^{c}-\Gamma_{ac,b}^c+\Gamma_{ab}^{d}\Gamma_{cd}^c-\Gamma_{ac}^d\Gamma_{bd}^c-\tilde{\Gamma}_{ab,c}^{c}+\tilde{\Gamma}_{ac,b}^c-\tilde{\Gamma}_{ab}^{d}\tilde{\Gamma}_{cd}^c+\tilde{\Gamma}_{ac}^d\tilde{\Gamma}_{bd}^c\right).\\
\end{split}\nonumber
\end{equation}
Now we get that
\begin{equation}\label{scalar1}
\begin{split}
&|\left(g^{ab}-\tilde{g}^{ab}\right)\left(\Gamma_{ab,c}^{c}-\Gamma_{ac,b}^c+\Gamma_{ab}^{d}\Gamma_{cd}^c-\Gamma_{ac}^d\Gamma_{bd}^c\right)|\left(x\right)\\
\leq&C_{14}\left(\|g\|_{C^2\left(M^n,\tilde{g}\right)}+\|g\|_{C^1\left(M^n,\tilde{g}\right)}\cdot\|g\|_{C^1\left(M^n,\tilde{g}\right)}\right)|u|\left(x\right).\\
\end{split}
\end{equation}
It follows from the formula of $\Gamma_{ab}^c-\tilde{\Gamma}_{ab}^c$ in Lemma \ref{decayofV} that
\begin{equation}\label{scalar2}
\begin{split}
&|\Gamma_{ab,c}^{d}-\tilde{\Gamma}_{ab,c}^{d}|\left(x\right)\\
=&|\tilde{\nabla}_{c}\left(\Gamma_{ab}^d-\tilde{\Gamma}_{ab}^d\right)|\left(x\right)\\
\leq&C_{15}\left(\|g\|_{C^1\left(M^n,\tilde{g}\right)}\cdot|\tilde{\nabla}u|\left(x\right)+\|g\|_{C^2\left(M^n,\tilde{g}\right)}\cdot|u|\left(x\right)+|\tilde{\nabla}^2g|\left(x\right)+|\tilde{\nabla}u|\left(x\right)\right)\\
\end{split}
\end{equation}
and
\begin{equation}\label{scalar3}
\begin{split}
&|\Gamma_{ab}^{c}\Gamma_{pq}^d-\tilde{\Gamma}_{ab}^{c}\tilde{\Gamma}_{pq}^d|\left(x\right)\\
=&|\Gamma_{ab}^{c}\Gamma_{pq}^d-\tilde{\Gamma}_{ab}^{c}\Gamma_{pq}^d+\tilde{\Gamma}_{ab}^{c}\Gamma_{pq}^d-\tilde{\Gamma}_{ab}^{c}\tilde{\Gamma}_{pq}^d|\left(x\right)\\
\leq&|\left(\Gamma_{ab}^{c}-\tilde{\Gamma}_{ab}^{c}\right)\Gamma_{pq}^d|\left(x\right)+|\tilde{\Gamma}_{ab}^{c}\left(\Gamma_{pq}^d-\tilde{\Gamma}_{pq}^d\right)|\left(x\right)\\
\leq&C_{16}\|g\|_{C^1\left(M^n,\tilde{g}\right)}\cdot|\Gamma_{ab}^{c}-\tilde{\Gamma}_{ab}^{c}|\left(x\right)+C_{16}|\Gamma_{pq}^d-\tilde{\Gamma}_{pq}^d|\left(x\right)\\
\leq&C_{17}\left(\|g\|_{C^1\left(M^n,\tilde{g}\right)}+1\right)\left(\|g\|_{C^1\left(M^n,\tilde{g}\right)}\cdot|u|\left(x\right)+|\tilde{\nabla}g|\left(x\right)\right)\\
\end{split}
\end{equation}
Together with (\ref{scalar1}), (\ref{scalar2}), (\ref{scalar3}) and Lemma \ref{decayingestimate}, we have that
\begin{equation}
R\left(t\right)+n\left(n-1\right)\leq\left(\epsilon_{18}+\frac{\epsilon_{18}}{\sqrt{t}}\right)e^{-\delta\rho},~~~\text{for}~~~t\in[0,1].\nonumber
\end{equation}
In particular, at $t=1,$
\begin{equation}
\bar{R}\left(\cdot,1\right)+n\left(n-1\right)\leq2\epsilon_{18}e^{-\delta\rho}.\nonumber
\end{equation}

Meanwhile, due to (\ref{scalar2}), (\ref{scalar3}), (\ref{gclosetogtilde}) and the diffeomorphism invariance
\begin{equation}
|\bar{R}ic+(n-1)\bar{g}|_{\bar{g}}=|Ric\left(\Phi^{\ast}_{t}g\right)+(n-1)\Phi^{\ast}_{t}g|_{\Phi^{\ast}_{t}g}=|Ric\left(g\left(t\right)\right)+(n-1)g\left(t\right)|_{g\left(t\right)},\nonumber
\end{equation}
we see that
\begin{equation}
|\bar{R}ic+(n-1)\bar{g}|_{\bar{g}}\leq \epsilon_{19}e^{-\sigma_5t}e^{-\gamma\rho}.\nonumber
\end{equation}
Hence
\begin{equation}
|\bar{R}ic+(n-1)\bar{g}|^2_{\bar{g}}\leq \epsilon_{20}e^{-2\sigma_5t}e^{-2\gamma\rho}.\nonumber
\end{equation}
Therefore, under $NRF,$ it satisfies
\begin{equation}
\frac{\p}{\p t}\bar{S}\leq\Delta \bar{S}-2(n-1)\bar{S}+\epsilon_{20}e^{-2\sigma_5t}e^{-2\gamma\rho},~~~\text{}~~~t\in[1,\infty)\nonumber
\end{equation}
with $\bar{S}(\cdot,1)\leq2\epsilon_{18}e^{-\delta\rho}$.
Then by the same arguments as those in the proof of Lemma \ref{decayofV} and note that the constant $-2\left(n-1\right)$ before the zero order term $\bar{S}$ will make $B$ in ODE more positive hence will bring no trouble, we have
\begin{equation}\label{decayofSNRF}
|\bar{S}|\leq \epsilon_{21} e^{-\sigma_6 t} e^{-\delta \rho}.
\end{equation}

For any $x\in M$ and $s\in (0, t)$, let $\gamma(s)=\Phi_s\left(x\right)$ be an integral curve of $V$, together with Lemma \ref{decayofV},
we see that
\begin{equation}\label{mapdifference}
d_{\tilde {g}}(\Phi_t\left(x\right), x)\leq \int^t_0 |\dot \gamma\left(s\right)|ds\leq C_{18} \int^t_0 |V\left(x,s\right)|ds\leq \epsilon_{22}e^{-\delta \rho},
\end{equation}
where $d_{\tilde {g}}(\cdot,\cdot)$ denotes the distance function in $M$ with respect to metric $\tilde g$. Hence, we get
\begin{equation}
\begin{split}
|R\left(t\right)+n\left(n-1\right)|\left(x\right)&\leq |\bar{S}|\left(x\right)+|R\left(g\left(t\right)\right)-R\left(\bar g\left(t\right)\right)|\left(x\right)\\
&\leq |\bar{S}|\left(x\right)+|R\left(g\left(t\right)\right)(x)-R\left(g\left(t\right)\right)(\Phi_t\left(x\right))|\\
&\leq |\bar{S}|\left(x\right)+C_{19}\|\tilde{\nabla}R\|_{C^0\left(M^n,\tilde{g}\right)}\cdot d_{\tilde {g}}(\Phi_t\left(x\right), x)\\
&\leq \epsilon''e^{-\bar{\sigma}t}e^{-\delta \rho\left(x\right)},
\end{split}\nonumber
\end{equation}
where we used (\ref{decayofSNRF}), (\ref{bigtimeonederivative}) and (\ref{mapdifference}) for $t\in[1,\infty)$ in the above inequality.
\end{proof}

Now, we can show that
\begin{prop}\label{monotonicityvolume}
Let $g\left(t\right),$ $t\in[0,\infty)$ be a solution to the NRDF (\ref{RicciDeTurckflow}). Then we have
\begin{equation}
\mathcal{V}\left(g\left(t\right)\right)=\mathcal{V}\left(g\left(0\right)\right)-\int^t _0 \int_M \left(R\left(g\left(s\right)\right)+n\left(n-1\right)\right)d\mu_g ds,\nonumber
\end{equation}
where $d\mu_g $ is volume element with respect to metric $g(s)$. Moreover, $\mathcal{V}\left(g\left(t\right)\right)$ is non-increasing in $t$ and \begin{equation}\label{volumelimit}
\lim_{t\rightarrow \infty}\mathcal{V}\left(g\left(t\right)\right)=0.\nonumber
\end{equation}
\end{prop}
\begin{proof}
Let $\Omega$ be any compact domain in $M$ with smooth boundary, then by a direct computation, under NRDF (\ref{RicciDeTurckflow}), we have

$$
\frac{d}{dt}\int_\Omega \left(\sqrt{|g|}-\sqrt{|\tilde g|}\right)dx=-\int_\Omega \left(R+n(n-1)\right)d\mu_g +\int_{\partial \Omega}\langle V,\nu\rangle_gd\sigma,
$$
where $\nu$ is the outward unit normal vector of $\partial \Omega$, hence, we obtain
\begin{equation}\label{integralformulavolume1}
\begin{split}
&\int_\Omega \left(\sqrt{|g(t)|}-\sqrt{|\tilde g|}\right)dx-\int_\Omega \left(\sqrt{|g(t_0)|}-\sqrt{|\tilde g|}\right)dx\\
=&-\int^t_{t_0} \int_\Omega \left(R\left(g(s)\right)+n(n-1)\right)d\mu_g ds +\int^t_{t_0}\int_{\partial \Omega}\langle V,\nu\rangle_gd\sigma ds\\
\end{split}
\end{equation}

Combine (\ref{integralformulavolume1}) with Lemma \ref{decayofV} and Lemma \ref{decayofscalarcurvature}, and let $\Omega$ exhaust the whole manifold $M$ we get
\begin{equation}\label{integralformulavolume}
\mathcal{V}\left(g\left(t\right)\right)=\mathcal{V}\left(g\left(t_0\right)\right)-\int^t _{t_0} \int_M \left(R\left(g(s)\right)+n\left(n-1\right)\right)d\mu_g ds.
\end{equation}

Let $t_0=1$ and $\Omega_{\rho_0}=\{x\in M:\rho\left(x\right)\leq\rho_0\}.$ For any small $\eta>0,$ on one hand, subtract (\ref{integralformulavolume1}) from (\ref{integralformulavolume}), we see that there is a large $\rho_0>0$ such that

\begin{equation}
\begin{split}
|\int_{M\setminus \Omega_{\rho_0}} \left(\sqrt{|g(t)|}-\sqrt{|\tilde g|}\right)dx|&\leq|\int_{M\setminus \Omega_{\rho_0}}\left(\sqrt{|g(1)|}-\sqrt{|\tilde g|}\right)dx| \\
&+|\int^t_{1} \int_{M\setminus \Omega_{\rho_0}}\left(R\left(g\left(s\right)\right)+n\left(n-1\right)\right)d\mu_g ds|\\
& +|\int^t_{1}\int_{\partial \Omega_{\rho_0}}\langle V,\nu\rangle_gd\sigma ds|\\
&\leq \frac\eta2
\end{split}\nonumber
\end{equation}
where we have already used Lemma \ref{decayingestimate}, Lemma \ref{decayofscalarcurvature} and Lemma \ref{decayofV} in the above inequality. On the other hand, there is a large $T_0\geq1$ which depends only on $\rho_0$ and $\eta$ so that for any $t\geq T_0$ we have
$$
|\int_{ \Omega_{\rho_0}}(\sqrt{|g(t)|}-\sqrt{|\tilde g|})dx|\leq \frac\eta2.
$$

Hence for any small $\eta>0$ there is a large $T_0$ which depends only on $\eta$ so that for any $t\geq T_0$

$$
|\mathcal{V}(g(t))|\leq \eta,
$$
which implies that
\begin{equation}
\lim_{t\rightarrow \infty}\mathcal{V}(g(t))=0.\nonumber
\end{equation}
Thus we finish to prove the proposition.
\end{proof}

Now, we can prove our main results.
\begin{proof}[Proof of Theorem \ref{main1}]
Consider NRDF (\ref{RicciDeTurckflow}) and NRF (\ref{Ricciflow}) starting from $g=g_0$, and let $g\left(t\right)$ and $\bar{g}\left(t\right)$ be the solution to (\ref{RicciDeTurckflow}) and (\ref{Ricciflow}) respectively. By Proposition \ref{monotonicityvolume}, we obtain

$$
\mathcal{V}(g)=\mathcal{V}(g_0)\geq \mathcal{V}(g(t))\geq0,
$$
that is, $$\mathcal{V}(g)\geq0.$$
\end{proof}

\begin{proof}[Proof of Theorem \ref{rigidity}]
If equality holds, we have $$
\mathcal{V}(g)= \mathcal{V}(g(t))=0,~~~\text{for}~~~~t\in[0,\infty),
$$
which implies
$$
\int_{M^n}(R(g(t))+n(n-1))d\mu_g = 0,
$$
together with the fact that $R(g(t))\geq -n(n-1)$, we get that on $M$ and for all $t\in[0,\infty)$
$$
R\left(g(t)\right)=-n(n-1),
$$
which means
$$
R\left(\bar{g}(t)\right)=-n(n-1).
$$
By the evolution equation of $R$ under NRF (\ref{Ricciflow}), we see that
$$
Ric\left(\bar{g}(t)\right)=-(n-1)\bar{g}(t)
$$
for all $t\in[0,\infty).$
Thus the initial metric $g$ is an Einstein metric, which means that the NRDF (\ref{RicciDeTurckflow}) is just acting by diffeomorphisms. According to the NRDF
\begin{equation}
\left\{
  \begin{array}{ll}
    \frac{\p}{\p t}g_{ij}=\nabla_{i}V_j+\nabla_{j}V_i\\
    g\left(\cdot,0\right)=g=g_0\\
    V_{j}=g_{jk}g^{pq}\left(\Gamma_{pq}^{k}-\tilde{\Gamma}_{pq}^{k}\right),
  \end{array}
\right.\nonumber
\end{equation}
and the NRF
\begin{equation}
\left\{
  \begin{array}{ll}
    \frac{\p}{\p t}\bar{g}_{ij}=0\\
    \bar{g}\left(\cdot,0\right)=g=g_0,\\
  \end{array}
\right.\nonumber
\end{equation}
we have
\begin{equation}\label{phiinfty}
g=g_0=\bar{g}(t)=\Phi^{\ast}_{t}g(t).\nonumber
\end{equation}

Since
\begin{equation}\label{diffeomorphism}
\left\{
  \begin{array}{ll}
    \frac{\p}{\p t}\Phi_t\left(x\right)=-V\left(\Phi_t\left(x\right),t\right),\\
    \Phi_0=id,
  \end{array}
\right.\nonumber
\end{equation}
and under the NRDF, for all $t\in[0,\infty),$ the $1$-form $V$ satisfies
\begin{equation}
|V\left(t\right)|\leq \epsilon' e^{-\tilde{\sigma} t} e^{-\delta\rho},\nonumber
\end{equation}
hence
\begin{equation}
|\Phi_t-id|\leq \frac{\epsilon'}{\tilde{\sigma}}\left(1-e^{-\tilde{\sigma} t}\right)e^{-\delta\rho}\nonumber
\end{equation}
which implies that there exists some smooth diffeomorphism $\Phi_{\infty}$ of $M^n$ satisfying $\Phi_{t}\rightarrow\Phi_{\infty}$ in $C^{\infty}\left(M^n, M^n\right)$ as $t\rightarrow\infty$ and $\Phi_{\infty}$ extends continuously to some diffeomorphism on $\bar{M}$ and $\Phi_{\infty}|_{\p M}=id.$ Therefore
\begin{equation}\label{phiinfty}
g=\Phi_{\infty}^{\ast}\tilde{g}\nonumber
\end{equation}
and thus we finish to prove Theorem \ref{rigidity}.
\end{proof}


\begin{thebibliography}{99}
\bibitem{AST} I.Agol, P.A.Storm and W.P.Thurston, {\it Lower bounds on volumes of hyperbolic Haken 3-manifolds},
J. Am. Math. Soc. 20(4), 1053-1077 (2007)
\bibitem{Ba} E.Bahuaud, {\it Ricci flow of conformally compact metrics}, Ann.I.H.Poincar$\acute{e}$-AN 28 (2011) 813-835
\bibitem{Bam} R.Bamler, {\it Stability of hyperbolic manifolds with cusps under Ricci flow}, arXiv:1004.2058
 \bibitem{Br} H.Bray, {\it The Penrose inequality in general relativity and volume comparison theorems involving scalar curvature}, PhD thesis, Stanford University (1997) arXiv:0902.3241
\bibitem{BC} S.Brendle and O.Chodosh, {\it A volume comparison theorem for asymptotically hyperbolic manifolds}, to appear in Comm. Math. Phys., arXiv:1305.6628
\bibitem{BMW} E.Bahuaud, R.Mazzeo and E.Woolgar, {\it Renormalized volume and the evolution of APES}, arXiv:1307.4788
\bibitem{CLN} B.Chow, P.Lu and L.Ni, {\it Hamilton's Ricci flow}, Lectures in Contemporary Mathmatics 3, AMS, (1998)
\bibitem{D} E.Delay, {\it Essential spectrum of the Lichnerowicz Laplacian on two tensors on asymptotically hyperbolic manifolds}, Journal of Geometry and Physics 43 (2002) 33-44
\bibitem{EH} K.Ecker and G.Huisken, {\it Interior estimates for hypersurfaces moving by mean curvature}, Invent. math. 105, 547-569 (1991)
\bibitem{G}  R.Gicquaud, {\it $\acute{E}$tude de quelques probl$\grave{e}$mes d'analyse et de g$\acute{e}$ometrie sur les vari$\acute{e}$t$\acute{e}$s asymptotiquement hyperboliques}, PhD thesis, Universit$\acute{e}$ Montpellier 2, 2009.
\bibitem{HQS} X.Hu, J.Qing and Y.G.Shi, {\it Regularity and rigidity of asymptotically hyperbolic manifolds}, Adv. Math. 230 (2012), 2332-2363

\bibitem{Lee} J.M.Lee, {\it Fredholm operators and Einstein metrics on conformally compact manifolds}, Memoirs of the American Mathematical Society, A.M.S., 2006.
\bibitem{LT} G.Liao and L.F.Tam, {\it On the heat equation for harmonic maps from non-compact manifolds}, Pacific J. Math. Vol 153, (1992), no.1, 129-145
\bibitem{LY} H.Z.Li and H.Yin, {\it On stability of the hyperbolic space form under the normalized Ricci flow}, International Mathematics Research Notices, Vol. 2010, No. 15, pp. 2903-2924
\bibitem{MT} P.Z.Miao and L.F.Tam, {\it On the volume functional of compact manifolds with boundary with constant scalar curvature}, Calc. Var. (2009) 36:141-171
\bibitem{P1} G.Perelman, {\it The entropy formula for the Ricci flow and its geometric applications}, arXiv:math.DG/0211159
\bibitem{P2} G.Perelman, {\it Ricci flow with surgery on three-manifolds}, arXiv:math.DG/0303109
\bibitem{QSW} J.Qing, Y.G.Shi and J.Wu, {\it Normalized Ricci flows and conformally compact Einstein metrics}, Calc. Var. Partial Differential Equations 46 (2013), no.1-2, 183-211.
\bibitem{S} R.Schoen, {\it Variational theory for the total scalar curvature functional for Riemannian
metrics and related topics},  Topics in calculus of variations, Lecture Notes in Math., 1365(1989), Springer, Berlin, 120-154.
\bibitem{SSS} O.C.Schn$\ddot{u}$rer, F.Schulze and M.Simon, {\it Stability of hyperbolic space under Ricci flow}, Communications in Analysis and Geometry, Vol.19, 2011, No.5, 1023-1047
\bibitem{Shi} W.X.Shi, {\it Deforming the metric on complete Riemannian manifolds}, J. Differential Geom, 30, 223-301 (1989)
\bibitem{Sim} M.Simon, {\it Deforming Lipschitz metrics into smooth metrics while keeping their curvature operator non-negative}, Geometric evolution equations, Contemp. Math., 367, Providence, RI: Amer. Math. Soc., 2005, 167-179

\end{thebibliography}
\end{document}